\documentclass[11pt,reqno]{amsart}
\usepackage{amsaddr}
\usepackage[usenames,dvipsnames]{xcolor}
\RequirePackage[OT1]{fontenc}

\usepackage[colorlinks,citecolor=blue,urlcolor=blue,linkcolor=blue,linktocpage=true]{hyperref}

\usepackage[]{amsmath}
\usepackage{amssymb,amsthm,amsfonts,amsbsy,latexsym,amsxtra}

\usepackage{graphicx}

\usepackage{todonotes}
\usepackage{appendix}
\usepackage{url}

\allowdisplaybreaks[3]

\usepackage[usenames,dvipsnames]{xcolor}            
\usepackage{graphicx}
\usepackage{bbm,bm}
\usepackage{comment}
\usepackage{mathtools}
\usepackage{placeins}
\usepackage[labelfont={sl}]{caption}
\usepackage[labelfont={normalfont}]{subcaption}

\usepackage{amssymb,amsfonts,amsmath,amsthm,amscd,dsfont,mathrsfs}

\usepackage{url}

\usepackage{enumerate}

\numberwithin{equation}{section}

\newtheorem{theorem}{Theorem}[section]

\newtheorem{lemma}[theorem]{Lemma}
\newtheorem{corollary}[theorem]{Corollary}

\theoremstyle{definition}

\newcommand{\E}{\mathbf{E}}
\newcommand{\R}{\mathbb{R}}
\newcommand{\M}{\mu}
\newcommand{\sP}{\mathcal{P}}
\newcommand{\N}{\mathbb{N}}

\renewcommand{\P}{\mathbf{P}}
\newcommand{\Prob}[1]{\mathbf P\left\{#1\right\}}
\renewcommand{\emptyset}{\varnothing}
\newcommand{\XX}{\mathbb{X}}
\newcommand{\QQ}{\mathbb{Q}}
\newcommand{\F}{\mathcal{F}}
\newcommand{\Nb}{\mathbf{N}}

\newcommand{\diff}{{\,\mathrm d}}
\newcommand{\ldiff}{{\mathrm d}}

\renewcommand{\preceq}{\prec}

\DeclareMathOperator{\Var}{Var}

\newtheorem*{theorem*}{Theorem}

\allowdisplaybreaks

\makeatletter
\@namedef{subjclassname@2020}{\textup{2020} Mathematics Subject Classification}
\makeatother

\title[Gaussian approximation in a random minimal directed spanning tree]{Gaussian approximation for rooted edges in a random minimal directed spanning tree}
\author[C. Bhattacharjee]{Chinmoy Bhattacharjee}
\address{Department of Mathematics, University of Luxembourg, 4364 Esch-sur-Alzette, Luxembourg}
\email{chinmoy.bhattacharjee@uni.lu}

\subjclass[2020]{Primary: 60F05, 60D05, Secondary: 05C80, 60G55}
\keywords{Spanning tree, minimal point, Stein's method, stabilization, Poisson
	process, central limit theorem.}

\begin{document}

\maketitle

\begin{abstract} We study the total $\alpha$-powered length of the rooted edges in a random minimal directed spanning tree - first introduced in Bhatt and Roy (2004) - on a Poisson process with intensity $s \ge 1$ on the unit cube $[0,1]^d$ for $d \ge 3$. While a Dickman limit was proved in Penrose and Wade (2004) in the case of $d=2$, in dimensions three and higher, Bai, Lee and Penrose (2006) showed a Gaussian central limit theorem when $\alpha=1$, with a rate of convergence of the order $(\log s)^{-(d-2)/4} (\log \log s)^{(d+1)/2}$. In this paper, we extend these results and prove a central limit theorem in any dimension $d \ge 3$ for any $\alpha>0$. Moreover, making use of recent results in Stein's method for region-stabilizing functionals, we provide presumably optimal non-asymptotic bounds of the order $(\log s)^{-(d-2)/2}$ on the Wasserstein and the Kolmogorov distances between the distribution of the total $\alpha$-powered length of rooted edges, suitably normalized, and that of a standard Gaussian random variable.
\end{abstract}

\begingroup
\hypersetup{linktocpage=false}
\setcounter{tocdepth}{2}
\endgroup

\section{Introduction and main results} The notion of a random minimal directed spanning tree was first introduced by Bhatt and Roy in \cite{BR04} to model certain transmission or drainage networks \cite{RR01} where, unlike the model considered in \cite{G61}, the signals/waves can travel only in certain specific directions. A typical motivating example in two dimensions, as considered in \cite{BR04}, is when a source radio transmitter placed at the origin emits a signal which can be received only by receivers which are positioned north-easterly with respect to the origin, i.e., in the first quadrant. Each of the transmitters at the frontier in turn amplifies the signal and transmits it to the receivers that are placed in the first quadrant with respect to its position. The resulting network, when several such receivers/transmitters are placed randomly in the first quadrant, has a tree structure which is directed north-easterly with its root at the origin. Such a graph is called a random minimal directed spanning tree (MDST).

The added feature of directionality gives rise to many interesting properties in an MDST. As for minimal spanning trees, one of the main objects of interest is the total $\alpha$-powered length, which is the Euclidean length raised to the power $\alpha>0$, of all the edges. Distributional approximation results for the sum of $\alpha$-powered length of all the edges in an MDST with a different partial ordering on the points than the one considered here was proved in \cite{PW10}, where it was shown that for $d \ge 2$, with vertices taken to be a Poisson process on $[0,1]^d$ with intensity $s \ge 1$, one obtains a Gaussian limit for small $\alpha$ while for large $\alpha$, one has an additional independent and possibly non-Gaussian part in the limit as $s \to \infty$; see also \cite{PW06}. 

In this paper, we consider a related statistic, the total $\alpha$-powered length of all the rooted edges, i.e., all the edges with one end at the origin. This was first studied in \cite{BR04} in two dimensions, where the existence of a distributional limit was proved. Soon after, Penrose and Wade \cite{PW04} identified the limiting distribution and showed a Dickman convergence as $s \to \infty$ for the total $\alpha$-powered length of rooted edges in an MDST on a Poisson (or a Binomial) process on $[0,1]^2$ with intensity $s \ge 1$ (respectively, with $s$ points for $s \in \N$) and $\alpha>0$. The question in dimensions three and higher was partially addressed in \cite{BLP06} where, unlike in two dimensions, a Gaussian central limit theorem was shown when $\alpha=1$. The case for a general $\alpha>0$ in dimensions $d \ge 3$ remained elusive.

In this paper, we aim to fill this gap. In Theorem~\ref{thm:Pareto}, we show that the total $\alpha$-powered length of the rooted edges in an MDST on a Poisson process on $[0,1]^d$, $d \ge 3$, with intensity $s \ge 1$, suitably normalized, has a Gaussian limit as $s \to \infty$ for any $\alpha>0$. Our proof uses a completely different approach based in stabilization theory and Stein's method. Indeed, we obtain a stronger result in the form of a quantitative central limit theorem providing presumably optimal rates of convergence, where, by analogy with the usual Berry-Esseen type results, we say a rate of Gaussian convergence is presumably optimal when it is of the order of inverse of the standard deviation of the statistic.

\subsection{Notation} We write $\R_+:=[0,\infty)$. For $x \in \R$ we write $x^+:=\max\{x,0\}$. For an integer $n \in \N$, we denote by $[n]:=\{1,\dots,n\}$. For real numbers $x ,y$, we write $x\wedge y$ and $x\vee y$ to denote the minimum and maximum, respectively, of $x$ and $y$. Throughout, $\|x\|$ stands for the usual $L^2$-norm of a point $x \in \R^d$. For $x=(x_1,\dots,x_d)\in \XX:=[0,1]^d$, let
$[0,x]:=[0,x_1] \times\cdots\times [0,x_d]$, and denote the
volume of $[0,x]$ by
$
	|x|:=\prod_{i=1}^d x_i
$.
For $I \subseteq [d]$, we write $x^{(I)}$ for the subvector $(x_i)_{i \in I}$ of $x$. Finally, for $k \in [d-1]$, we denote $I_k=[k]$ and $J_k=[d] \setminus I_k$. For two functions $f,g:\R_+ \to \R$ with $g \ge 0$, we write $f(s)=\mathcal{O}(g(s))$ to mean that the limit $\limsup_{s \to \infty} |f(s)|/g(s)$ is bounded, while $f(s) \simeq g(s)$ means that $f(s)-g(s)=\mathcal{O}(\log^{d-3} s)$.

\subsection{Model and main results}
We now explicitly describe our model. Let $\XX:=[0,1]^d$ be the $d$-dimensional unit cube for some integer $d \ge 2$. Let $0$ stand for the origin. We say a point
$x \in \R^d$ dominates a point $y \in \R^d$ if
$x-y\in \R_+^d\setminus\{0\}$, and write $x\succ y$, or equivalently, $y \prec x$. For $n \in \N$ and a collection of $n+1$ distinct vertices $\mathcal{V}=\{0,x^{(1)}, \dots, x^{(n)}\}$ in $\XX$, define the admissible edge set $E$ of directed edges as
$$
E:= \{(x,y): x, y \in \mathcal{V}, x\not=y, x \prec y\}.
$$
Consider the collection $\mathscr{G}$ of graphs $G$ with vertex set $\mathcal{V}$ and edge set $E_G \subseteq E$ with the property that for any $i \in [n]$, the vertex $x^{(i)}$ is connected to the origin by a path constructed from edges in $E_G$, i.e., either $(0,x^{(i)}) \in E_G$, or there exists distinct $i_1, \dots, i_m \in [n]$ with $m\in \N$ such that $(0,x^{(i_1)}) \in E_G$, $(x^{(i_m)},x^{(i)}) \in E_G$ and $(x^{(i_l)}, x^{(i_{l+1})}) \in E_G$ for all $1 \le l \le m-1$, where, by convention, the final inclusion is trivial when $m=1$.  

A \textit{minimal directed spanning tree} with vertex set $\mathcal{V}$ is a graph $T \in \mathscr{G}$ that minimizes $\sum_{e \in E_G} l(e)$ over all $G \in \mathscr{G}$, where $l(e)$ denotes the usual Euclidean length of an edge $e$, i.e.,
$$
\sum_{e \in E_T} l(e) = \min_{G \in \mathscr{G}} \sum_{e \in E_G} l(e).
$$
It is straightforward to see that any such $T$ is necessarily a tree (see Figure~\ref{fig:mdst}).
\begin{figure}[t!]
	\begin{center}
		\begin{tikzpicture}[scale = 5]
		\draw (0,0) rectangle (1,1);
		\node at (1.15,.98) {$(1,1)$};
		\node at (-.1,-.05) {$(0,0)$};
		\foreach \Point in {(0,0),(.24,.65),(.39, .54),(.4,.39),(.5,.6),(.55, .8), (.85,.9), (.30, .55),(.62, .35),(.75, .3),(.33,.86),(.72,.56),(.75,.78),(.88,.35),(.45,.25)}{
			\draw[fill,gray] \Point  circle [radius = .01];};
		\foreach \Point in {(.1,.8),(.15,.3),(.25,.2),(.25,.2),(.9,.07), (.6,.2)}{
			\draw[fill] \Point  circle [radius = .012];
		};
		\draw[thick] (0, 0) to (.1,.8);
		\draw[thick] (0, 0) to (.15,.3);
		\draw[thick] (0, 0) to (.25,.2);
		\draw[thick] (0, 0) to (.9,.07);
		\draw[thick] (0, 0) to (.6,.2);
		\draw[] (.6,.2) to (.62, .35);
		\draw[] (.6,.2) to (.75, .3);
		\draw[] (.25,.2) to (.4,.39);
		\draw[] (.15,.3) to (.24,.65);
		\draw[] (.15,.3) to (.30, .55);
		\draw[] (.15,.3) to (.39, .54);
		\draw[] (.39, .54) to (.5,.6);
		\draw[] (.1,.8) to (.33,.86);
		\draw[] (.5,.6) to (.55, .8);
		\draw[] (.62, .35) to (.72,.56);
		\draw[] (.72,.56) to (.75,.78);
		\draw[] (.75,.78) to (.85,.9);
		\draw[] (.25,.2) to (.45,.25);
		\draw[] (.75, .3) to (.88,.35);
		\end{tikzpicture}
		\caption{An MDST in the two dimensional unit cube $[0,1]^2$ (the point configuration includes all points in the figure except the origin). The dark points are the minimal points. The random variable $\mathscr{L}_0^{\alpha}$ is the $\alpha$-powered sum of the lengths of the thick black edges.\label{fig:mdst}}
	\end{center}
\end{figure}
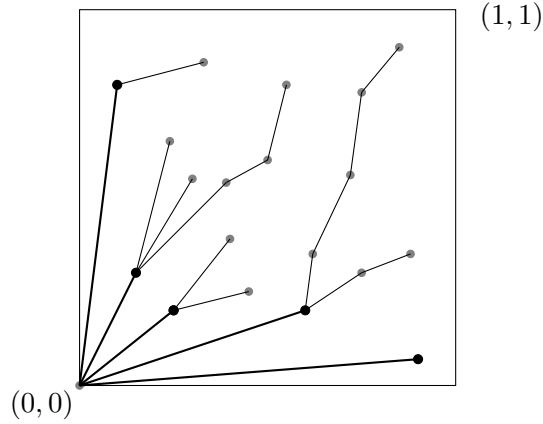

Let $\mu$ be a locally finite simple point configuration (we will interchangeably interpret $\mu$ as a point process or a point set) in $\XX \setminus \{0\}$ such that the MDST with vertex set $\{0\} \cup \mu$ is unique. Let $\mu^{min} \subseteq \mu$ denote the subset of vertices that are connected to the origin by an edge in the MDST, we call these points the \textit{minimal points} in the MDST on $\mu$. It is not hard to see that these are exactly the points in $\mu$ that do not dominate any other point in $\mu$, i.e., the minimal points are exactly the Pareto optimal points in $\mu$; for more details, see \cite{Ba20,BDHT05,FN20}. 

Let $\QQ$ be the Lebesgue measure on $\XX$ and for $s \ge 1$, let $\sP_s$ be a Poisson process with intensity measure $s \QQ$. An MDST on the vertex set $\{0\} \cup \sP_s$ is almost surely unique. In the rest of the paper, we consider this random MDST. For $\alpha>0$, let
\begin{equation}
\label{eq:ParetoPoints}
\mathscr{L}_0^{\alpha}\equiv \mathscr{L}_0^{\alpha}(\sP_s):=\sum_{x \in \sP_s^{min}} \|x\|^\alpha
\end{equation}
denote the sum of the $\alpha$-powered lengths of the rooted edges (see Figure~\ref{fig:mdst}), where $\sP_s^{min}$ stands for the set of minimal points in the MDST on $\sP_s$. In this paper, we concern ourselves with the distributional limit and a quantitative CLT for $\mathscr{L}_0^{\alpha}$. Our first result establishes the asymptotic behaviour of the mean and the variance of $\mathscr{L}_0^{\alpha}$. 

\begin{theorem}\label{thm:MV} For $d \ge 2$, $s > 1$ and $\alpha>0$,
	\begin{enumerate}[(a)]
		\item
	\begin{equation*}
		\E \mathscr{L}_0^{\alpha} = \frac{d}{\alpha (d-2)!} \log^{d-2} s + \mathcal{O}(\log^{d-3} s),
	\end{equation*}
\item 
\begin{equation*}
\Var{\mathscr{L}_0^{\alpha}} = \frac{1}{2\alpha(d-2)!} w(d,\alpha) \log^{d-2} s + \mathcal{O}(\log^{d-3} s),
\end{equation*}
where
\begin{multline}\label{eq:w}
w(d,\alpha):= d-2d \int_{\XX } \frac{b_1^\alpha }{(1+|b|)^{2}} \diff b \\
+ 2\sum_{k =1}^{d-1} k \binom{d}{k}\int_{\XX } b_1^{\alpha} \left(\frac{1}{(|b^{(I_k)}| + |b^{(J_k)}|-|b|)^{2}}  - \frac{1}{(|b^{(I_k)}| + |b^{(J_k)}|)^{2}}\right)\diff b
\end{multline}
satisfies $0<\inf_{\alpha>0} w(d,\alpha)\le \sup_{\alpha>0} w(d,\alpha)<\infty$.
\end{enumerate}
\end{theorem}

\noindent$\underline{Remark}:$ We note here that the assertions in Theorem~\ref{thm:MV} were proved in the case of $\alpha=1$ in \cite{BLP06}. In accordance with Theorem~2 therein, following computation similar to those in Section~3 in \cite{BLP06} to further simplify the integrals in \eqref{eq:w}, we can restate our variance estimate as follows: For $d \ge 2$, $s > 1$ and $\alpha>0$,
\begin{equation*}
	\Var{\mathscr{L}_0^{\alpha}} = \left(\frac{d}{2\alpha (d-2) !}-\gamma_{d}^{(\alpha)} +2 \sum_{k=1}^{d-1} k \binom{d}{k} h_{k}^{(\alpha)}\right) \log^{d-2} s + \mathcal{O}(\log^{d-3} s),
\end{equation*}
where
$$
\gamma_{d}^{(\alpha)}=\frac{d}{\alpha (d-2) !} \int_{0}^{1} v_{1}^\alpha \diff v_{1} \int_{0}^{1}  \left(1+v_{1} v_2\right)^{-2}\frac{\left(-\log v_2\right)^{d-2}}{(d-2) !} \diff v_2,
$$
and
$$
h_{1}^{(\alpha)}=\frac{1}{2\alpha(d-2) !} \int_{0}^{1} w_{1}^\alpha \diff w_{1} \int_{0}^{1}  \left(\left(w_{1}+w_{2}-w_{1} w_{2}\right)^{-2}-\left(w_{1}+w_{2}\right)^{-2}\right)  \frac{\left(-\log w_{2}\right)^{d-2}}{(d-2) !}\diff w_{2}
$$
while for $2 \le k \le d-1$,
\begin{multline*}
	h_{k}^{(\alpha)}=\frac{1}{2\alpha(d-2) !}  \int_{0}^{1} u_1^{\alpha-1}\diff u_{1} \int_{0}^{u_{1}} \diff w_{1} \int_{0}^{1} \left(\left(w_{1}+w_{2}-w_{1} w_{2}\right)^{-2}-\left(w_{1}+w_{2}\right)^{-2}\right) \\
	 \times \frac{\left(-\log w_{1}+\log u_{1}\right)^{k-2}}{(k-2) !} \frac{\left(-\log w_{2}\right)^{d-k-1}}{(d-k-1) !} \diff w_{2}.
\end{multline*}
Taking $\alpha=1$ reproduces the bound in \cite[Theorem~2]{BLP06}.

\vspace{.2cm}

To state our second main result, we need to introduce two metrics on the space of probability distributions. The \textit{Wasserstein distance} between (the distributions of) two real-valued random
variables $X$ and $Y$ is given by
$$
d_W(X,Y):= \sup_{h \in \operatorname{Lip}_1} |\E\; h(X) - \E \; h(Y)|,
$$
where $\operatorname{Lip}_1$ denotes the class of all Lipschitz
functions $h: \R \to \R$ with Lipschitz constant at most one. The
\textit{Kolmogorov distance} between (the distributions of) $X$ and $Y$ is given by
$$
d_K(X,Y):= \sup_{t \in \R} |\Prob{X \le t} - \Prob{Y \le t}|.
$$
In the following result, we derive non-asymptotic bounds on the Wasserstein and Kolmogorov
distances between suitably normalized $\mathscr{L}_0^{\alpha}$
and a standard Gaussian random variable, denoted by $N$ throughout the sequel.

\begin{theorem}\label{thm:Pareto}
	For $d \ge 3$ and $s >1$, let $\sP_s$ be a Poisson process on $[0,1]^d$ with intensity measure
	$s\QQ$, where $\QQ$ is the Lebesgue measure, and for $\alpha>0$, let $\mathscr{L}_0^{\alpha}$ be as in \eqref{eq:ParetoPoints}. Then there exists a constant $C \in (0,\infty)$ depending only on $\alpha$ and $d$ such that for all $ s>1$,
	\begin{displaymath}
	\max \left\{d_W\left(\frac{\mathscr{L}_0^{\alpha} - \E \mathscr{L}_0^{\alpha}}{\sqrt{\Var
			\mathscr{L}_0^{\alpha}}},N\right),
	d_K\left(\frac{\mathscr{L}_0^{\alpha} - \E \mathscr{L}_0^{\alpha}}{\sqrt{\Var \mathscr{L}_0^{\alpha}}},N\right)\right\} \le
	\frac{C}{\log^{(d-2)/2} s}.
	\end{displaymath}
\end{theorem}

\vspace{.1cm}

\noindent$\underline{Remarks}:$ \begin{enumerate}[(a)]
	\item We note here that in the setting of Theorem~\ref{thm:Pareto}, a weaker rate of convergence or the order $(\log s)^{-(d-2)/4} (\log \log s)^{(d+1)/2}$ in the Kolmogorov distance was shown in \cite{BLP06} for $\alpha=1$.
	\item Theorem~\ref{thm:Pareto} proves a Gaussian convergence as $s \to \infty$ in dimensions $d \ge 3$ for $\mathscr{L}_0^{\alpha}$. In contrast, in two dimensions, $\mathscr{L}_0^{\alpha}$ converges to the so-called Dickman distribution. This points out a change in the distributional behaviour of the minimal points as one goes from dimension two to three and above. The idea behind obtaining a Dickman limit in two dimensions is that most of the minimal points lie close to the axes and one can look only at the length of their projections on the two axes rather than the length of the edges themselves. When considering only the projections on the $x$-axis, say $X_1 \ge X_2 \ge \cdots$, one obtains that $X_{i+1}/X_i \sim \mathbb{U}[0,1]$ for all $i \ge 1$, where $\mathbb{U}[0,1]$ denotes a uniform random variable on the interval $[0,1]$, and they are independent. This gives rise to the Dickman limit, since a standard Dickman random variable $D$ has the representation
		$$
		D =_d \sum_{i=1}^\infty \prod_{j=1}^{i} U_j,
		$$
		where $U_j \sim \mathbb{U}[0,1]$, $j \in \N$ are independent and $=_d$ denotes equality in distribution. Such a property does not seem to hold in dimensions three and higher.
	\item  In Theorem~\ref{thm:MV}(b), we show that the variance of the statistic $\mathscr{L}_0^{\alpha}$ is exactly of the order $\log^{d-2} s$ for any $d \ge 3$. Hence, the upper bound in Theorem~\ref{thm:Pareto} is presumably of optimal order. 
	\item  Finally, we note that our results can potentially be extended to the setup of Binomial processes by proving a version of Theorem~\ref{thm:KolBd} below for such a  process.
\end{enumerate}
\vspace{.1cm}

Notice by Theorem~\ref{thm:MV}(a) that for $d= 3$, $s > 1$ and $\mathscr{L}_0^{\alpha}$ as in Theorem~\ref{thm:Pareto},
\begin{equation}\label{eq:cstbd}
\left|\frac{\mathscr{L}_0^{\alpha} - \E \mathscr{L}_0^{\alpha}}{\sqrt{\Var
		\mathscr{L}_0^{\alpha}}} - \frac{\mathscr{L}_0^{\alpha} -  \frac{d}{\alpha (d-2)!} \log^{d-2} s}{\sqrt{\Var
		\mathscr{L}_0^{\alpha}}} \right| \le \frac{C}{\sqrt{\log s}}
\end{equation}
for some constant $C \in (0,\infty)$ depending only on $\alpha>0$. Now, let $X$ and $Y$ be two real-valued random variables with $X=Y+a$ for some $a>0$. Then it is easy to see that
$
d_W(X,Y) \le a.
$
On the other hand, by the triangle inequality and the fact that the density of the standard Gaussian distribution is uniformly bounded by $1/\sqrt{2\pi}$, we have
$$
d_K(X,N)= \sup_{t \in \R} |\Prob{Y+a \le t} - \Prob{N +a\le t} + \Prob{N +a\le t} -\Prob{N \le t}| \le d_K(Y,N) + \frac{a}{\sqrt{2\pi}},
$$
which by symmetry yields
$$
|d_K(X,N) - d_K(Y,N)| \le \frac{a}{\sqrt{2\pi}}.
$$
Combining these two facts with \eqref{eq:cstbd}, we have the following corollary to Theorem~\ref{thm:Pareto}.

\begin{corollary} For $d =3$ and $\mathscr{L}_0^{\alpha}$ as in Theorem~\ref{thm:Pareto}, there exists a constant $C \in (0,\infty)$ depending only on $\alpha$ such that for all $s>1$,
	$$
	\max \left\{d_W\left(\frac{\mathscr{L}_0^{\alpha} - (3/\alpha) \log s}{\sqrt{\Var
			\mathscr{L}_0^{\alpha}}},N\right), d_K\left(\frac{\mathscr{L}_0^{\alpha} - (3/\alpha) \log s}{\sqrt{\Var
			\mathscr{L}_0^{\alpha}}},N\right) \right\} \le
	\frac{C}{\sqrt{\log s}}.
	$$
\end{corollary}

We now briefly discuss some of the important ingredients that the proofs of Theorems~\ref{thm:MV} and \ref{thm:Pareto} rely on. To prove Theorem~\ref{thm:MV}(a), we first provide an estimate of the mean of the sum over the minimal points in $\sP_s$ of certain weight functions, where the weights are functions of a few coordinates. When the weights are all identically equal to one, it is well-known (see e.g.\ \cite{BDHT05}) that for $d \ge 2$,
$$
\E\, |\sP_s^{min}|=\mathcal{O}(\log^{d-1} s).
$$
 If instead one takes the weights to be the norms of the minimal point as is the case in $\mathscr{L}_0^{\alpha}$, the problem boils down to estimating $\E \sum_{x \in \sP_s^{min}} x_1^\alpha$ and $\E \sum_{x \in \sP_s^{min}} (x_1 x_2)^{(1 \wedge \alpha)/2}$ (by an application of Lemma~\ref{lem:secord} below), where $x_i$ is the $i$-th coordinate of $x$ for $i \in [d]$. It turns out that by considering weights that are powers of one of the coordinates, the order of the expectation goes down by one logarithmic factor compared to the case when the weights are identically equal to one (see Theorem~\ref{thm:short}). Generally, consider for any $k \in [d-1]$, weight functions of minimal points of the form $\prod_{i=1}^k x_i^{\alpha_i}$, for some $\alpha_1,\dots, \alpha_k>0$. In this case, the expectation of the sum of the weights  over the minimal points is of the order $\log^{d-k-1} s$. We prove this fact in the following result. 

\begin{theorem}[Short version of Theorem~\ref{thm:mean}]\label{thm:short} For $d \ge 2$, $s>1$, $k \in [d-1]$, and $\alpha_1, \dots, \alpha_k>0$,
	\begin{equation*}
		s\int_{\XX} \Big(\prod_{i=1}^k x_i^{\alpha_i} \Big) e^{-s |x|} \diff x=\mathcal{O}(\log^{d-k-1} s).
	\end{equation*}
\end{theorem}

In Theorem~\ref{thm:MV}(b), we estimate the variance of $\mathscr{L}_0^{\alpha}$ with an exact leading order term. This is arguably the most crucial part of the paper and involves some delicate estimates. 

Finally, to prove the quantitative bounds in Theorem~\ref{thm:Pareto}, we make use of some recent results in \cite{BM21} which provides non-asymptotic bounds on the Wasserstein and Kolmogorov distances for Gaussian approximation of stabilizing functionals.

The rest of the paper is organized as follows. In Section~\ref{sec:regst} we present the tools from Stein's method and stabilization functionals in the form of Theorem~\ref{thm:KolBd} from \cite{BM21} that we utilize to provide our bounds. We obtain precise estimates of the mean and variance of $\mathscr{L}_0^{\alpha}$ in Section~\ref{sec:mv}, proving Theorem~\ref{thm:MV}. Finally, in Section~\ref{sec:Pareto}, we prove Theorem~\ref{thm:Pareto}.

\section{Bounds for sums of region-stabilizing functionals}\label{sec:regst} The random variable $\mathscr{L}_0^{\alpha}$ can be thought of as a sum of certain functionals, whose value at a particular point $x \in \XX$ depend only on the point configuration $\sP_s$ in some small neighbourhood of $x$. Such functionals are known as \textit{stabilizing functionals}. They were utilized in the context of Gaussian approximation starting with the works \cite{PY01,PY03} and were further advanced in \cite{BX06,PY05,Yu15}. In the relevant literature, one usually considers such functionals on semi-metric spaces where the `stabilization region' is taken to be a ball. But, in our example, this turns out to be vastly suboptimal. In \cite{BM21}, this problem was addressed by introducing a new and more general notion of \textit{region-stabilizing functionals} building upon the work \cite{LSY19}. In this section, we recall a bound from therein which we will use to prove Theorem~\ref{thm:Pareto}.

For $(\XX,\F)$ a Borel space, $\QQ$ a $\sigma$-finite
measure on $(\XX, \F)$ and $s \ge 1$, let $\sP_s$ be a Poisson
process with intensity measure $s\QQ$. Let $\Nb$ stand for the family of $\sigma$-finite counting
measures $\M$ on $\XX$ equipped with the smallest $\sigma$-algebra
$\mathscr{N}$ that makes the maps $\M \mapsto \M(A)$ measurable for
all $A \in \F$. 
For $\mu \in \Nb$, let $\M_A$ denote the
restriction of $\M$ onto a set $A\in\F$, i.e., $\M_A(B):=\int_\XX \mathds{1}_{A \cap B} \M(\diff x)$ for all $B \in \F$. We write $\M_1\leq\M_2$ for $\M_1,\M_2\in\Nb$,
if $\M_2-\M_1$ is non-negative. Let $(\xi_s)_{s \ge 1}$ be a collection of \emph{score functions} which are Borel measurable functions mapping each
pair $(x,\M) \in \XX \times \Nb$ to a real number.
Consider the random variable
\begin{equation}
\label{eq:hs}
H_s \equiv H_s(\sP_s):= \sum_{x \in \sP_s} \xi_s(x,\sP_s), \quad s \ge 1.
\end{equation}
Below, we list the assumptions required for a Gaussian approximation result for $H_s$. We denote by $0$ the zero counting measure.
\begin{enumerate}
	\item[(A0)]\textit{Monotonicity:} For $s \ge 1$, if
	$\xi_s(x, \M_1)=\xi_s(x, \M_2)$ for some $\M_1,\M_2 \in \Nb$ with
	$0 \neq \M_1\leq \M_2$, then
	\begin{equation*}
	\xi_s(x, \M_1)=\xi_s(x, \M) \quad 
	\text{for all} \, \M\in\Nb \text{ with  } \; \M_1\leq \M\leq \M_2.
	\end{equation*}
\end{enumerate}

Let $\delta_x$ stand for the Dirac
measure at $x\in\XX$. Recall, $\mu_A$ denotes the restriction of $\mu \in \Nb$ onto $A \in \F$. In the rest of the paper, we interpret elements in $\Nb$ as measures and for $\mu_1, \mu_2 \in \Nb$, we write $\mu_1 + \mu_2$ for the sum of the two measures.
\begin{enumerate}
	\item[(A1)] \textit{Stabilization region:} For all $s \ge 1$, there
	exists a map $R_s:\XX \times \Nb\to \F$ such that
	\begin{enumerate}[(1)]
		\item $$
		\{\mu \in \Nb : y\in R_s(x, \M+\delta_x)\} \in \mathscr{N}
		\quad \text{for all $x,y \in \XX$}
		$$
		and,
		$$
		\Prob{y\in R_s(x, \sP_s+\delta_x)}\, \text{ and } \,
		\Prob{\{y_1, y_2\}\subseteq  R_s(x, \sP_s+\delta_x)}
		$$
		are measurable functions of $(x,y) \in \XX^2$ and
		$(x,y_1,y_2) \in \XX^3$ respectively,
		\item the map $R_s$ is monotonically decreasing in the second argument,
		i.e.\
		$$
		R_s(x,\M_1) \supseteq R_s(x,\M_2),
		\quad \M_1 \leq \M_2,\; x\in\M_1,
		$$ 
		\item for all $\mu\in\Nb$ and $x\in\mu$, if $\mu_{R_s(x,\mu)} \neq 0$, then $(\mu+\delta_y)_{R_s(x,\mu +\delta_y)} \neq 0$ for all
		$y \notin R_s(x,\mu)$,  
		\item for all $\M\in\Nb$ and $x\in\M$,
		\begin{displaymath}
			\xi_{s}\big(x,\M\big)
			=\xi_{s}\big(x,\M_{R_{s}(x,\M)}\big).
		\end{displaymath}
	\end{enumerate}
\end{enumerate}

\begin{enumerate}
	\item[(A2)] \textit{$L^{4+p}$-norm:} There exists a $p \in (0,1]$
	such that, for all $\mu\in\Nb$ with $\mu(\XX) \le 7$,
	\begin{displaymath}
	\Big\|\xi_{s}\big(x, \sP_{s}+\delta_x+\mu\big)\Big\|_{4+p}
	\leq M_{s,p}(x), \quad s\geq1,\; x \in \XX,
	\end{displaymath}
	where $M_{s,p} : \XX \to \R$, $s\geq1$, are measurable
	functions and $\|\cdot\|_{4+p}$ denotes the $L^{4+p}$-norm. For notational convenience, in the sequel we will write $M_s$
	instead of $M_{s,p}$, and generally drop $p$ from all subscripts.
\end{enumerate}

Let $r_{s}: \XX \times \XX \to [0,\infty]$ be a non-zero measurable
	function such that
	\begin{equation}
	\label{eq:Rs}
	\Prob{y \in R_s(x, \sP_s +\delta_x)} \le  e^{-r_{s}(x,y)}, \quad x, y \in \XX.
	\end{equation}

For $p$ as in (A2) and $\lambda:=p/(40+10p)$, define the functions
\begin{gather}
\label{eq:g}
g_{s}(y) :=s \int_{\XX} e^{-\lambda r_{s}(x, y)} \;\QQ(\ldiff x), \quad h_s(y) :=s \int_{\XX} M_{s}(x)^{4+p/2}e^{-\lambda r_{s}(x, y)} \;\QQ(\ldiff x)\\ 
\label{eq:g5}
G_s(y) := \widetilde{M}_{s,p}(y) +
\tilde h_s(y)\big(1+g_s(y)^4\big), \quad y\in\XX,
\end{gather}
where 
$$
\widetilde{M}_{s}(y):=\max\{M_{s}(y)^2,M_{s}(y)^4\} \quad \text{and} \quad \tilde h_s(y)=\max\{h_s(y)^{2/(4+p/2)}, h_s(y)^{4/(4+p/2)}\}.
$$ Next, letting
\begin{equation}
\label{eq:g2s}
q_{s}(x,y):=s \int_\XX \P\Big\{\{x,y\} 
\subseteq R_s\big(z, \sP_s +\delta_z\big)\Big\} \;\QQ(\ldiff z), \quad x,y \in \XX,
\end{equation} 
for
$\gamma>0$, define
\begin{equation}
\label{eq:fa}
f_\gamma(y):=f_\gamma^{(1)}(y)+f_\gamma^{(2)}(y)+f_\gamma^{(3)}(y),
\quad y\in\XX,
\end{equation}
where for $y \in \XX$,
\begin{align}
\label{eq:fal}
f_\gamma^{(1)}(y)&:=s \int_\XX G_s(x) e^{- \gamma r_{s}(x,y)}
\;\QQ(\ldiff x), \notag \\
f_\gamma^{(2)} (y)&:=s \int_\XX G_s(x) e^{- \gamma r_{s}(y,x)} 
\;\QQ(\ldiff x), \nonumber\\
f_\gamma^{(3)}(y)&:=s \int_{\XX} G_s(x) q_{s}(x,y)^\gamma \;\QQ(\ldiff x).
\end{align}
Finally, let
\begin{equation}
\label{eq:p}
\kappa_s(x):= \Prob{\xi_{s}(x, \sP_{s}+\delta_x) \neq 0},\quad
x\in\XX. 
\end{equation}
Below, we write $\QQ f:=\int_\XX f(x) \QQ(\ldiff x)$ for an integrable function $f : \XX \to \R$.

\begin{theorem}[Theorem~2.1, \cite{BM21}]\label{thm:KolBd}
	Assume that $(\xi_s)_{s \ge 1}$ satisfy conditions (A0)--(A2) and
	let $H_s$ be as in \eqref{eq:hs}. Then, for $p$ as in (A2) and
	$\theta:=p /(32+4 p)$, 
	\begin{align*}
	d_{W}\left(\frac{H_s-\E H_s}{\sqrt{\Var H_s}},  N\right) 
	&\leq C \Bigg[\frac{\sqrt{s  \QQ f_\theta^2}}{\Var H_s}
	+\frac{ s\QQ ((\kappa_s+g_{s})^{2\theta}G_s)}{(\Var H_s)^{3/2}}\Bigg],
	\end{align*}
	and
	\begin{multline*}
	d_{K}\left(\frac{H_s-\E H_s}{\sqrt{\Var H_s}},
	N\right) 
	\leq C \Bigg[\frac{\sqrt{s  \QQ f_\theta^2}
		+ \sqrt{s  \QQ f_{2\theta}}}{\Var H_s}
	+\frac{\sqrt{ s\QQ ((\kappa_s+g_{s})^{2\theta}G_s)}}{\Var H_s}
	+\frac{ s\QQ ((\kappa_s+g_{s})^{2\theta} G_s)}{(\Var H_s)^{3/2}}\\
	 +\frac{( s\QQ ((\kappa_s+g_{s})^{2\theta} G_s))^{5/4}
		+ ( s\QQ ((\kappa_s+g_{s})^{2\theta} G_s))^{3/2}}{(\Var H_s)^{2}}\Bigg]
	\end{multline*}
	for all $s\geq1$, where
	$C \in (0,\infty)$ is a constant depending only on $p$.
\end{theorem}

\section{Estimating the mean and variance}\label{sec:mv} In this section, we estimate the mean and variance of the random variable $\mathscr{L}_0^{\alpha}$.  
For results when $\alpha$ is zero so that $\mathscr{L}_0^{\alpha}=|\sP_{s}^{min}|$ counts the number of minimal points in $\sP_{s}$, see \cite{BDHT05} and references therein. When $\alpha>0$, the problem of estimating the moments becomes much more involved. The case when $\alpha=1$ was considered in \cite{BLP06}. The goal of this section is to achieve good estimates for any $\alpha>0$. Throughout, $C$ stands for a generic finite positive constant whose value might change from one line to the next. Since $\QQ$ is fixed to be the Lebesgue measure, for economy of notation, we omit $\QQ$ in integrals and write $\diff x$ instead of $\QQ(\ldiff x)$.

\subsection{Mean}
By the Poisson empty space formula, 
$$
\Prob{x \in (\sP_s + \delta_x)^{min}} =\Prob{\sP_s ( [0,x])=0}=e^{-s|x|}, \quad x \in \XX. 
$$
Hence, by the Mecke formula,
$$
\E \mathscr{L}_0^\alpha=s \int_\XX\|x\|^\alpha e^{-s|x|} \diff x.
$$
In Lemma~\ref{lem:secord}, for $x \in \XX$ and $\alpha >0$ we will show that $\big|\|x\|^\alpha - \sum_{i=1}^d x_i^\alpha \big| \le C \sum_{i \not = j \in [d]} (x_i x_j)^{(1 \wedge \alpha)/2}$ for some constant $C$.
In the following result we demonstrate that $\E \sum_{x \in \sP_s^{min}} x_1^\alpha$ is of the order $\log^{d-2} s$ for $d \ge 2$, while $\E \sum_{x \in \sP_s^{min}} (x_1x_2)^\alpha $ has order $\log^{d-3} s$ for $d \ge 3$. Combining this with Lemma~\ref{lem:secord} will then yield Theorem~\ref{thm:MV}(a).

We will often use the following estimates: for any $\alpha>0$, $s > 1$ and $\delta\ge 1$, 
\begin{equation}\label{eq:loges}
\int_0^s |\log w|^{\delta} w^{\alpha-1} \diff w=\alpha^{-1} s^{\alpha} \log^\delta s + \mathcal{O}(s^{\alpha} \log^{\delta-1} s).
\end{equation}
Also notice that for $\delta \ge 0$, $\alpha > -1$ and $\beta>0$,
$$
\int_{0}^\infty w^\alpha |\log w|^{\delta}  e^{-\beta w}\diff w \le \int_0^1 |\log w|^\delta \diff w + \int_1^\infty w^{\delta+\alpha} e^{-\beta w}\diff w \le  \int_0^1 |\log w|^\delta \diff w + \frac{\Gamma (\delta+\alpha+1)}{\beta^{\delta+\alpha+1}},
$$
where $\Gamma$ is the Gamma function. Since any non-negative power of logarithm is integrable on $(0,1]$,
for all $\delta \ge 0$, $\alpha > -1$ and $\beta>0$,
\begin{equation}
	\label{eq:Gamma}
	\int_{0}^\infty w^{\alpha} |\log w|^{\delta} e^{-\beta w}\diff w <\infty. 
\end{equation} 

\begin{theorem}\label{thm:mean} For $d \ge 2$, $s>1$, $k \in [d-1]$, and $\alpha_1,\dots,\alpha_k, \beta, \nu>0$, $\delta \ge 0$ and $\tau>-1$,
	\begin{equation*}
	s\int_{\XX} \Big(\prod_{i=1}^k x_i^{\alpha_i}\Big) (s|x|)^{\tau}  \Big|\log (\nu s |x|)\Big|^\delta  e^{-\beta s |x|} \diff x=\mathcal{O}(\log^{d-k-1} s).
	\end{equation*}
Moreover, for $d \in \N$ and $k=d$ with $\alpha=\min_{i \in [d]} \alpha_i$,
\begin{equation*}
	s\int_{\XX} \Big(\prod_{i=1}^d x_i^{\alpha_i}\Big) (s|x|)^{\tau}  \Big|\log (\nu s |x|)\Big|^\delta  e^{-\beta s |x|} \diff x=\mathcal{O}(s^{-\alpha}\log^{d-1} s).
\end{equation*}
\end{theorem}
\begin{proof}
	Without loss of generality, assume $\nu=1$ and $\alpha_i=\alpha>0$ for all $i \in [d]$, where $\alpha=\min_{i \in [d]} \alpha_i$. First, fix $d \ge 2$ and $k \in [d-1]$. The
	derivation here loosely follows those used to calculate the mean of
	the number of minimal points in
	\cite[Sec.~2]{BDHT05}. Changing variables $u=s^{1/d}x$ in the first equality, and
	letting $z_i=-\log u_i$ for $i\in [d]$ in the second, we obtain
	\begin{align*}
	&s\int_{\XX} |x^{(I_k)}|^{\alpha} (s|x|)^{\tau}  \Big|\log (s |x|)\Big|^\delta  e^{-\beta s |x|} \diff x\\
	&=s^{-k \alpha/d}\int_{[0, s^{1/d}]^d} |u^{(I_k)}|^{\alpha} |u|^{\tau} \big|\log |u|\big|^\delta e^{-\beta|u|}   \diff u\\
	&=s^{-k \alpha/d} \int_{\big[-\frac{\log s}{d}, \infty\big)^d}
	\Big|\sum_{j=1}^d z_j\Big|^\delta \exp \bigg\{-\beta e^{- \sum_{j=1}^d z_j} - (1+\tau) \sum_{j=1}^d z_j - \alpha \sum_{i=1}^k  z_i
	\bigg\}  \diff z.
	\end{align*}
	Next, change variables by letting $v=(v_1, \dots, v_d)$
	where $v_i:=z_i+\cdots+z_d$, $i\not=2,\cdots, k+1$ and $v_2=z_1, \dots ,v_{k+1}=z_{k}$. Note that the
	integrand is only a function of $v^{(I_{k+1})}=(v_1,\dots,v_{k+1})$ and the Jacobian for the transformation is one. Taking into account the
	integration bounds on each $z_i$, for each admissible $v^{(I_{k+1})}$, when $d \ge k+2$ we have 
	\begin{displaymath}
	- \frac{d-k-1}{d} \log s \le v_{k+2} \le v_1-\sum_{j=2}^{k+1}v_j + \frac{\log s}{d},
	\end{displaymath}
	and when $d \ge k+3$,
	\begin{displaymath}
	- \frac{d-i+1}{d} \log s \le v_i \le v_{i-1} + \frac{\log s}{d}, \quad k+3 \le i \le d.
	\end{displaymath}
	Hence, integrating with respect to $v_{k+2}, \dots, v_d$, we obtain
	\begin{align*}
	&	s\int_{\XX} |x^{(I_k)}|^{\alpha} (s|x|)^{\tau}  \Big|\log (s |x|)\Big|^\delta  e^{-\beta s |x|} \diff x\\
	&= \frac{s^{-k\alpha/d}}{(d-k-1)!} \int_{-\log s}^\infty \int_{-\frac{\log s}{d}}^{v_1 + \frac{d-1}{d} \log s} \int_{-\frac{\log s}{d}}^{v_1 -v_2 + \frac{d-2}{d} \log s} \dots \int_{-\frac{\log s}{d}}^{v_1 -\sum_{j=2}^{k} v_j+ \frac{d-k}{d} \log s} |v_1|^\delta\\
	& \qquad \times \left(\frac{d-k}{d} \log s + v_1-\sum_{j=2}^{k+1}v_j \right)^{d-k-1}  \exp\Big\{-\beta e^{- v_1} - (1+\tau) v_1 -\alpha \sum_{j=2}^{k+1} v_j\Big\}   \diff v_{k+1}\cdots\diff v_2 \diff v_1.
	\end{align*}
	Now, substituting $w_i=e^{-v_i}$, $i \in [k+1]$ for the first equality and $\bar w_2=s^{(d-1)/d}w_2, \bar w_j=s^{-1/d}w_j$ for $3 \le j \le k+1$ in the second yield
	\begin{align}\label{eq:4}
	&s\int_{\XX} |x^{(I_k)}|^{\alpha} (s|x|)^{\tau}  \Big|\log (s |x|)\Big|^\delta  e^{-\beta s |x|} \diff x \nonumber\\
	&= \frac{s^{-k\alpha/d}}{(d-k-1)!} \int_0^s \int_{w_1 s^{-\frac{d-1}{d}}}^{s^{1/d}}\int_{\frac{w_1}{w_2} s^{-\frac{d-2}{d}}}^{s^{1/d}}  \dots \int_{\frac{w_1}{w_2 \dots w_k} s^{-\frac{d-k}{d}}}^{s^{1/d}}   w_1^{\tau}  \Big| \log w_1 \Big|^\delta e^{-\beta w_1} \nonumber\\
	&\qquad\qquad \times \left(\log (s^{(d-1)/d}w_2) + \sum_{j=3}^{k+1} \log (s^{-1/d}  w_j) - \log w_1 \right)^{d-k-1} \Big(\prod_{j=2}^{k+1} w_j\Big)^{\alpha-1} \diff w_{k+1}\cdots\diff w_2 \diff w_1 \nonumber\\
	&=\frac{s^{-\alpha}}{(d-k-1)!} \int_0^s \int_{w_1}^{s} \int_{w_1/\bar w_2}^1 \dots \int_{w_1/(\bar w_2 \cdots \bar w_k)}^1 w_1^{\tau}\Big| \log w_1 \Big|^\delta  e^{-\beta w_1} \nonumber \\
	& \qquad \qquad \times  \left( \sum_{j=2}^{k+1}\log \bar w_j - \log w_1 \right)^{d-k-1} \Big(\prod_{j=2}^{k+1} \bar w_j\Big)^{\alpha-1} \diff \bar w_{k+1}\cdots\diff \bar w_2 \diff w_1 \nonumber\\
	&=\frac{s^{-\alpha}}{(d-k-1)!} \sum_{i=0}^{d-k-1} \binom{d-k-1}{i} \int_0^s \int_{w_1}^{s} \int_{w_1/\bar w_2}^1 \dots \int_{w_1/(\bar w_2 \dots \bar w_k)}^1 w_1^{\tau}\Big| \log w_1 \Big|^\delta e^{-\beta w_1} \nonumber\\
	& \qquad \qquad   \times \Big(\sum_{j=3}^{k+1}\log \bar w_j\Big)^i  (\log \bar w_2 - \log w_1)^{d-k-1-i}  \Big(\prod_{j=2}^{k+1} \bar w_j\Big)^{\alpha-1} \diff \bar w_{k+1}\cdots\diff \bar w_2 \diff w_1.
	\end{align}
	Fix $i \in\{0,1,\dots, d-k-1\}$. Since non-negative powers of logarithm are integrable on $(0,1]$, we have that 
	$$
	\int_{[0,1]^{k-1}} \Big|\sum_{j=3}^{k+1}\log \bar w_j\Big|^i \Big(\prod_{j=3}^{k+1} \bar w_j\Big)^{\alpha-1} \diff (\bar w_3,\dots,\bar w_{k+1})<\infty.
	$$
	On the other hand,
	\begin{align*}
		&\int_0^s \int_{w_1}^{s}
		w_1^{\tau}\Big| \log w_1 \Big|^\delta  e^{-\beta w_1}  (\log \bar w_2 - \log w_1)^{d-k-1-i}  \bar w_2^{\alpha-1} \diff \bar w_2 \diff w_1\\
		& = \sum_{j=0}^{d-k-1-i} \binom{d-k-1-i}{j}(-1)^j\int_0^s \int_{0}^{\bar w_2}
		w_1^{\tau}\Big| \log w_1 \Big|^\delta  e^{-\beta w_1}  (\log \bar w_2)^{d-k-1-i-j} (\log w_1)^{j}  \bar w_2^{\alpha-1} \diff w_1 \diff \bar w_2 \\
		&\le C \sum_{j=0}^{d-k-1-i} \binom{d-k-1-i}{j}(-1)^j\int_0^s
		(\log \bar w_2)^{d-k-1-i-j}  \bar w_2^{\alpha-1} \diff \bar w_2 =\mathcal{O}(s^{\alpha}(\log s)^{d-k-i-1}),
	\end{align*}
	where in the penultimate step we have used that since $\tau>-1$, by \eqref{eq:Gamma} there exists a constant $C \in (0,\infty)$ such that $\int_0^\infty w_1^{\tau} |\log w_1|^{j+\delta} e^{-\beta w_1} \le C$ for all $0 \le j \le d-k-1-i$, and the final step is due to \eqref{eq:loges}. Combining the above two estimates, summing over $i$ and applying \eqref{eq:4} yields the result when $d \ge 2$ and $k \in [d-1]$.
	
	Finally, for $d \in \N$ and $k=d$, arguing exactly as above, one obtains
	\begin{align*}
	&	s\int_{\XX} |x|^{\alpha} (s|x|)^{\tau}  \Big|\log (s |x|)\Big|^\delta  e^{-\beta s |x|} \diff x\\
	&\le s^{-\alpha} \int_{-\log s}^\infty (v_1+\log s)^{d-1} |v_1|^\delta  \exp\Big\{-\beta e^{- v_1} - (1+\tau+\alpha) v_1 \Big\}  \diff v_1.
\end{align*}
The second assertion now follows upon substituting $w_1=e^{-v_1}$ and applying Jensen's inequality followed by \eqref{eq:Gamma}.
\end{proof}

For $\alpha>0$, by equivalence of $L^2$ and $L^\alpha$-norms, the assertion in Theorem~\ref{thm:mean} with $k=1$ implies that for $d \ge 2$, $s>1$, $\alpha, \beta, \nu>0$, $\delta \ge 0$ and $\tau>-1$,
\begin{equation}\label{eq:3.1}
	s\int_{\XX} \|x\|^{\alpha} (s|x|)^{\tau}  \Big|\log (\nu s |x|)\Big|^\delta  e^{-\beta s |x|} \diff x=\mathcal{O}(\log^{d-2} s).
\end{equation}
In particular for $\alpha>0$,
$$
\E \mathscr{L}_0^{\alpha} =s\int_{\XX} \|x\|^\alpha e^{-s|x|} \diff x \le d^{\alpha/2} \sum_{i=1}^d s\int_{\XX} x_i^\alpha e^{-s|x|} \diff x=\mathcal{O}(\log^{d-2} s).
$$
But to obtain a good estimate for the variance, we need the exact leading order term of the mean. With a more careful computation in Theorem~\ref{thm:mean}, one obtains the following result. 

\begin{lemma}\label{lem:mex}
	For $d \ge 2$, $s>1$, $\alpha>0$, $\beta >0$ and $\tau >-1$,
	\begin{equation*}
	s\int_{\XX} x_1^\alpha (s|x|)^{\tau}  e^{-\beta s |x|} \diff x= \frac{\Gamma(1+\tau)}{\alpha \beta^{\tau +1} (d-2)!} \log^{d-2} s + \mathcal{O}(\log^{d-3} s).
	\end{equation*} 
\end{lemma}

\begin{proof}
	Arguing exactly as in Theorem~\ref{thm:mean} with $\delta=0$, one obtains
	\begin{equation}\label{eq:a}
	s\int_{\XX} x_1^\alpha (s|x|)^{\tau}  e^{-\beta s |x|} \diff x 
	\simeq \frac{s^{-\alpha}}{(d-2)!}  \int_0^s \int_{0}^{\bar w_2} w_1^{\tau} e^{- \beta w_1} (\log  \bar w_2)^{d-2}  \bar w_2^{\alpha-1} \diff w_1 \diff  \bar w_2.
	\end{equation}
	For $C=\int_0^\infty w_1^\tau e^{-\beta w_1} \diff w_1$, notice using \eqref{eq:loges} that
	\begin{equation}\label{eq:b}
	\int_0^{\sqrt{s}} \int_{0}^{\bar w_2} w_1^{\tau}  e^{-\beta w_1} |\log \bar w_2|^{d-2} \bar w_2^{\alpha-1} \diff w_1 \diff \bar w_2  \le C \int_0^{\sqrt{s}} |\log \bar w_2|^{d-2} \bar w_2^{\alpha-1} \diff \bar w_2 = \mathcal{O}(s^{\alpha/2} \log^{d-2} s).
	\end{equation}
	Also for $w_1 \ge 0$,  we have $w_1^\tau e^{-\beta w_1/2} \le C$ for some constant $C$. Thus, for $\bar w_2 \ge \sqrt{s},$
	$$
	\int_{0}^{\bar w_2} w_1^{\tau} e^{- \beta w_1} \diff w_1=\frac{\Gamma(1+\tau)}{\beta^{\tau+1}}+ \frac{e^{-\beta \sqrt{s}/2}}{\beta^{\tau+1}} \int_{\beta \bar w_2}^\infty w^\tau e^{-w+\beta \sqrt{s}/2} \diff w=\frac{\Gamma(1+\tau)}{\beta^{\tau+1}}+ \mathcal{O}(e^{-\beta \sqrt{s}/2}).
	$$
	Using this and \eqref{eq:loges}, we obtain
	\begin{multline*}
	 \frac{s^{-\alpha}}{(d-2)!} \int_{\sqrt{s}}^s \int_{0}^{\bar w_2} w_1^{\tau} e^{- \beta w_1} (\log \bar w_2)^{d-2} \bar w_2^{\alpha-1} \diff w_1 \diff\bar w_2\\
	  \simeq \frac{s^{-\alpha} \Gamma(1+\tau)}{\beta^{\tau +1}(d-2)!} \int_{\sqrt{s}}^s (\log \bar w_2)^{d-2} \bar w_2^{\alpha-1} \diff \bar w_2
	   \simeq \frac{\Gamma(1+\tau)}{\alpha \beta^{\tau +1}(d-2)!}\log^{d-2} s.
	\end{multline*}
	Combining this with \eqref{eq:b} yields the result by \eqref{eq:a}.
\end{proof}

\begin{lemma}\label{lem:secord}
	For $\alpha>0$ and $x \in \XX$, there exists a constant $C \in (0,\infty)$ depending only on $\alpha$ and $d$ such that
	$$\big|\|x\|^{\alpha} - \sum_{i=1}^d x_i^{\alpha}\big| \le C \sum_{i \not = j \in [d]} (x_i x_j)^{(1 \wedge \alpha)/2}.$$
\end{lemma}
\begin{proof}
	We first note that
	\begin{equation*}
	\Big| \|x\| - \sum_{i=1}^d x_i \Big| \le \frac{(\sum_{i=1}^d x_i)^2 - \sum_{i=1}^d x_i^2}{\sum_{i=1}^d x_i + \|x\|} \le \sum_{i \not = j \in [d]} \frac{x_ix_j}{x_i+x_j} \le  \frac{1}{2} \sum_{i \not = j \in [d]} \sqrt{x_ix_j}.
	\end{equation*}
	Now applying the mean value theorem along with the fact that $\sum_{i=1}^d x_i \ge \|x\|$, for $\alpha > 1$ we obtain
	\begin{equation}\label{eq:1}
	\Big|\|x\|^{\alpha} - \Big(\sum_{i=1}^d x_i\Big)^{\alpha}\Big| \le  \frac{\alpha}{2} \Big(\sum_{i=1}^d x_i \Big)^{\alpha-1} \sum_{i \not = j \in [d]} \sqrt{x_ix_j} \le \frac{\alpha d^{\alpha-1}}{2}\sum_{i \not = j \in [d]} \sqrt{x_ix_j}.
	\end{equation}
	Similarly, when $\alpha \in (0,1]$, 
	\begin{equation}\label{eq:2}
	\Big|\|x\|^{\alpha} - \Big(\sum_{i=1}^d x_i\Big)^{\alpha} \Big| \le  \frac{\alpha}{2} \sum_{i \not = j \in [d]} \frac{\sqrt{x_ix_j}}{\|x\|^{1-\alpha}} \le \frac{\alpha}{2} \sum_{i \not = j \in [d]} \frac{\sqrt{x_ix_j}}{(x_i^2 + x_j^2)^{(1-\alpha)/2}} \le \frac{\alpha}{2} \sum_{i \not = j \in [d]} (x_ix_j)^{\alpha/2}.
	\end{equation}
	When $\alpha>1$, noting that $(\sum_{i=1}^d x_i)^{\alpha} \ge \sum_{i=1}^d x_i^\alpha$, we have for some constants $C_1,C_2>0$ that 
	\begin{align*}
	\Big|\Big(\sum_{i=1}^d x_i\Big)^{\alpha} - \sum_{i=1}^d x_i^\alpha\Big| &\le \frac{\big(\sum_{i=1}^d x_i\big)^{\lceil \alpha \rceil} - \sum_{i=1}^d x_i^{\lceil \alpha \rceil}}{\big(\sum_{i=1}^d x_i\big)^{\lceil \alpha \rceil - \alpha}} \\
	&\le C_1 \sum_{i \not = j \in [d]} \sum_{l_1,l_2=1}^{\lfloor \alpha \rfloor} \frac{x_i^{l_1} x_j^{l_2}}{(x_ix_j)^{(\lceil \alpha \rceil - \alpha)/2}} \le C_2 \sum_{i \not = j \in [d]} \sqrt{x_i x_j}.
	\end{align*}
	Putting this together with \eqref{eq:1} yields the claim for $\alpha>1$. Finally, for $\alpha \in (0,1]$, let $k \in \N$ be such that $2^{k-1} \alpha \le 1 <2^k \alpha$. We will use induction on $k$ to show that there exists some constant $C_\alpha' \in (0,\infty)$ such that
	\begin{equation}\label{eq:3}
	\Big|\Big(\sum_{i=1}^d x_i\Big)^{\alpha} - \sum_{i=1}^d x_i^\alpha\Big| \le C_\alpha' \sum_{i \not = j \in [d]} (x_ix_j)^{\alpha/2}.
	\end{equation}
	When $k=1$, using that $(\sum_{i=1}^d x_i)^{2\alpha} \ge \sum_{i=1}^d x_i^{2\alpha}$, we have
	\begin{align*}
	\Big|\Big(\sum_{i=1}^d x_i\Big)^{\alpha} - \sum_{i=1}^d x_i^\alpha\Big| \le \frac{2 \sum_{i \not = j \in [d]} (x_i x_j)^\alpha + \sum_{i=1}^d x_i^{2\alpha} - \big(\sum_{i=1}^d x_i\big)^{2\alpha}}{\big(\sum_{i=1}^d x_i\big)^{\alpha} + \sum_{i=1}^d x_i^\alpha} \le \sum_{i \not = j \in [d]} \frac{(x_i x_j)^\alpha}{(x_ix_j)^{\alpha/2}}.
	\end{align*}
	Assume that \eqref{eq:3} holds for $k=l \in \N$. Then when $k=l+1$, by the induction hypothesis, arguing as above we have
	\begin{align*}
	\Big|\Big(\sum_{i=1}^d x_i\Big)^{\alpha} - \sum_{i=1}^d x_i^\alpha\Big| \le \frac{(2+C_{2\alpha}') \sum_{i \not = j \in [d]} (x_i x_j)^\alpha}{\big(\sum_{i=1}^d x_i\big)^{\alpha} + \sum_{i=1}^d x_i^\alpha} \le (1+ C_{2\alpha}'/2) \sum_{i \not = j \in [d]} (x_ix_j)^{\alpha/2}.
	\end{align*}
	This proves \eqref{eq:3}, which upon combining with \eqref{eq:2} yields the assertion for $\alpha \in (0,1]$.
\end{proof}

Putting together the assertion of Theorem~\ref{thm:mean} with $k=2$, Lemma~\ref{lem:mex} and Lemma~\ref{lem:secord}, we obtain the following corollary.

\begin{corollary}\label{cor:1}
	For any $d \ge 2$, $s > 1$, $\alpha, \beta>0$ and $\tau >-1$, 
	$$
	s\int_{\XX} \|x\|^\alpha (s|x|)^\tau e^{-\beta s |x|}\diff x= \frac{d \Gamma(1+\tau)}{\alpha \beta^{\tau+1}(d-2)!} \log^{d-2} s + \mathcal{O}(\log^{d-3} s).
	$$
\end{corollary}

In particular, taking $\beta=1$ and $\tau=0$, this implies Theorem~\ref{thm:MV}(a), i.e., that for $d \ge 2$ and $\alpha>0$,
\begin{equation*} 
\E \mathscr{L}_0^{\alpha} = s\int_{\XX} \|x\|^\alpha e^{-s |x|}\diff x= \frac{d}{\alpha (d-2)!} \log^{d-2} s + \mathcal{O}(\log^{d-3} s).
\end{equation*}

\subsection{Variance} This section is devoted to the proof of Theorem~\ref{thm:MV}(b) estimating of the variance of $\mathscr{L}_0^{\alpha}$. 
A precise estimate for the leading order term of the variance was obtained in \cite{BLP06} when $\alpha=1$. In Theorem~\ref{thm:MV}(b), we obtain such an estimate for any $\alpha>0$. 

For $\beta>0$, $s>0$, and $d\in\N$, define the function
$c_{\beta,s}: \XX \to \R_+$ as
\begin{equation}\label{eq:cdef}
c_{\beta,s} (y):=s\int_{\XX}\mathds{1}_{x\succ y}
e^{-\beta s |x|} \diff x.
\end{equation}
We recall the following result from \cite{BM21}.

\begin{lemma}[Lemma~3.1, \cite{BM21}]
	\label{lem:BM21}
	For $s>0$, there exists a constant $C\in (0,\infty)$ depending only on $d \in \N$ such that
	\begin{displaymath}
	c_{\beta,s}(y)\leq \frac{C}{\beta}
	e^{-\beta s|y|/2}\Big[1+\big|\log(\beta s|y|)\big|^{d-1}\Big], \quad y \in \XX.
	\end{displaymath}
\end{lemma}

The function $c_{\beta,s}$ satisfies the scaling property
\begin{equation*}
	c_{\beta,s}(x)=\beta^{-1} c_{1,\beta s}(x), \quad \beta>0,\; s>0.
\end{equation*}
This enables us to take $\beta=1$ without loss of 
generality.  In this article, we will consider a slightly generalized version of $c_{1,s}$. 
For $s >0$, $\delta \ge 0$, $\tau>-1$ and $d \in \N$, define the function $c_{\delta,
	\tau,s}: \XX \to \R_+$ as
\begin{equation}
	\label{eq:cal}
	c_{\delta,\tau,s} (y):=s\int_{\XX}\mathds{1}_{x\succ y}
	e^{-s |x|} \big|\log (s |x|) \big|^\delta (s|x|)^\tau \diff x,
\end{equation}
while, for $ \alpha >0$ and $k \in [d]$, define the function
$c_{\alpha,\delta,\tau, s}^{(k)}: \XX \to \R_+$ as
\begin{equation}
	\label{eq:cal'}
	c_{\alpha,\delta,\tau,s}^{(k)} (y):=s\int_{\XX} \mathds{1}_{x\succ y} |x^{(I_k)}|^\alpha
	e^{-s |x|} \big|\log (s |x|) \big|^\delta (s |x|)^\tau \diff x.
\end{equation}
The following lemma demonstrates the asymptotic behaviour
of $c_{\delta,\tau,s}$ and $c_{\alpha,\delta,\tau,s}^{(k)}$ for large $s$.

\begin{lemma}
	\label{lemma:c-bound}
	For $d \in \N$, $\delta \ge 0$, $\tau>-1$ and $s>0$, there exists a constant $C \in (0,\infty)$ depending only on $d$, $\delta$ and $\tau$ such that
	\begin{displaymath}
		c_{\delta,\tau,s}(y)\leq C
		e^{-s|y|/2}\Big[1+\big|\log(s|y|)\big|^{d-1}\Big], \quad y \in \XX.
	\end{displaymath}
	Further, for any $\alpha>0$ and $k \in [d]$,
	\begin{equation*}
		c_{\alpha,\delta,\tau,s}^{(k)}(y)\leq C'
		\frac{|y^{(I_k)}|^{\alpha'}}{(s|y|)^{k\alpha'}} e^{-s|y|/2}\Big[1+\big|\log(s|y|)\big|^{(d-k-1)^+ + (d-1)\mathds{1}_{k=d}}\Big], \quad y \in \XX
	\end{equation*}
for any $\alpha' \in (0,\alpha]$ for a constant $C' \in (0,\infty)$ that depends on $\alpha, \alpha',\delta, \tau$ and $d$.
\end{lemma}
\begin{proof}
	The first assertion is a slight modification of \cite[Lemma~3.1]{BM21}, with an additional logarithmic factor and the factor $(s|x|)^\tau$ in the integrand. This, however, doesn't change the proof, we demonstrate this in the proof of the second assertion, and refer to \cite[Lemma~3.1]{BM21} for a proof of the first one. 
	
	The
	derivation here is again motivated by those in
	\cite[Sec.~2]{BDHT05}. 
	We will consider the case when $k=d \in \N$ at the end of the proof. For $d \ge 2$, fix $k \in [d-1]$. Since $|x^{(I_k)}| \le 1$,
	$$
	c_{\alpha,\delta,\tau,s}^{(k)}(y) \le s\int_\XX \mathds{1}_{x \succ y} |x^{(I_k)}|^{\alpha'} e^{-s|x|}\big|\log (s |x|) \big|^\delta (s |x|)^\tau \diff x= c_{\alpha',\delta,\tau,s}^{(k)}(y). 
	$$
	Changing variables $u=s^{1/d}x$ in
	the definition of $c_{\alpha',\delta,\tau,s}^{(k)}$ to obtain the first equality, and
	letting $z_i=-\log u_i$, $i \in [d]$, in the second, for
	$y \in \XX$ we obtain
	\begin{multline}\label{eq:ctil}
		s^{k\alpha'/d} c_{\alpha',\delta,\tau,s}^{(k)}(y)
		=\int_{\times_{i=1}^d [s^{1/d}y_i, s^{1/d}]} |u^{(I_k)}|^{\alpha'} e^{-|u|} \big|\log (|u|) \big|^\delta |u|^\tau \diff u \\
	= \int_{\times_{i=1}^d \big[-d^{-1}\log s, -d^{-1}\log s -
			\log y_i\big]}
		\exp \bigg\{-e^{-\sum_{j=1}^d z_j} - (1+\tau)\sum_{j=1}^d z_j - \alpha' \sum_{j=1}^k z_j
		\bigg\}\Big|\sum_{j=1}^d z_j\Big|^\delta \diff z.
	\end{multline}
	As in Theorem~\ref{thm:mean}, we let $v=(v_1, \dots, v_d)$
	with $v_i:=z_i+\cdots+z_d$, $i\not=2,\cdots, k+1$ and $v_2=z_1, \dots ,v_{k+1}=z_{k}$. Taking into account the
	integration bounds on $z_i$, when $d \ge k+2$, we have 
	\begin{displaymath}
		v_1- \sum_{j=2}^{k+1}v_j - \bigg(- \frac{i-k-1}{d} \log s - \sum_{j=k+1}^{i-1} \log y_i\bigg)
		\le v_i \le - \frac{d-i+1}{d} \log s - \sum_{j=i}^d \log y_i,
		\quad k+2  \le i \le d.
	\end{displaymath}
	Thus, for each $k+2 \le i \le d$, the integration variable $v_i$
	belongs to an interval of length at most $(-\log (s|y|) + \log (s^{k/d}y^{(I_k)}) -
	v_1+\sum_{j=2}^{k+1}v_j)$. On the other hand, given $v_1 \in [-\log s, -\log (s|y|)]$, we note that for $2 \le j \le k+1$,
	$$
	-\log (s^{1/d}y_{j-1}) \ge v_j=v_1 - \sum_{i\neq j-1} z_i \ge v_1 + \log (s|y|) - \log (s^{1/d}y_{j-1}).
	$$
	In particular, $0 \le -\log (s|y|) + \log (s^{k/d}y^{(I_k)}) -
	v_1+\sum_{j=2}^{k+1}v_j \le -\log (s|y|) -
	v_1$. Hence, changing variables and bounding the integrals w.r.t.\ $v_{k+2}, \dots, v_d$ in the first step and substituting $w=e^{-v_1}$ in the last one, for $d \ge 2$ we have from \eqref{eq:ctil} that
	\begin{align*}
		&s^{k\alpha'/d} c_{\alpha',\delta,\tau,s}^{(k)}(y)\\
		&\le \int_{-\log s}^{-\log (s|y|)} \int_{v_1 + \log (s|y|) - \log (s^{1/d}y_1)}^{-\log (s^{1/d}y_1)}  \dots \int_{v_1 + \log (s|y|) - \log (s^{1/d}y_k)}^{-\log (s^{1/d}y_k)} \exp\Big\{-e^{-v_1} - (1+\tau)v_1 -\alpha' \sum_{j=2}^{k+1}v_j\Big\} \\
		&\qquad \qquad\qquad\qquad \times |v_1|^\delta \Big(-\log (s|y|) + \log (s^{k/d}y^{(I_k)}) -
		v_1+ \sum_{j=2}^{k+1}v_j\Big)^{d-k-1} \diff v_{k+1} \cdots \diff v_2\diff v_1\\
		&\le \int_{-\log s}^{-\log (s|y|)}  \Big(-\log (s|y|) -
		v_1\Big)^{d-k-1}|v_1|^\delta \exp\Big\{-e^{-v_1} - (1+\tau)v_1\Big\} \\
		& \qquad\qquad\qquad\qquad \times \prod_{j=2}^{k+1} \left[\int_{v_1 + \log (s|y|) - \log (s^{1/d}y_{j-1})}^{-\log (s^{1/d}y_{j-1})} e^{-\alpha' v_j} \diff v_j \right]\diff v_1\\
		&\le \frac{s^{k\alpha'/d} (y^{(I_k)})^{\alpha'} }{\alpha'(s|y|)^{k\alpha'}} \int_{-\log s}^{-\log (s|y|)}  \Big(-\log (s|y|) -
		v_1\Big)^{d-k-1} |v_1|^\delta e^{-
			k\alpha' v_1} \exp\Big\{-e^{-v_1} - (1+\tau)v_1\Big\}\diff v_1\\
		&=  \frac{s^{k\alpha'/d} (y^{(I_k)})^{\alpha'} }{\alpha'(s|y|)^{k\alpha'}}   \int_{s|y|}^{s} \Big(\log w-\log (s|y|)\Big)^{d-k-1} |\log w|^\delta w^{k\alpha'+\tau} e^{-w}
		\diff w.
	\end{align*}
	Applying Jensen's inequality, we obtain
	\begin{align*}
		c_{\alpha,\delta,\tau,s}^{(k)}(y) 
		& \le 2^{(d-k-2)^+} \frac{(y^{(I_k)})^{\alpha'}}{\alpha'(s|y|)^{k\alpha'}} e^{-s|y|/2} \bigg[ \big|\log (s|y|)\big|^{d-k-1} \int_{s|y|}^{s} |\log w|^\delta w^{k\alpha'+\tau} e^{-w/2}
		\diff w\\
		& \qquad \qquad\qquad\qquad\qquad\qquad\qquad\qquad\qquad\qquad+ \int_{s|y|}^s |\log w|^{d-k-1+\delta} w^{k\alpha'+\tau}e^{-w/2}\diff w \bigg].
	\end{align*}
	The result for $k \in [d-1]$ with $d \ge 2$ now follows by \eqref{eq:Gamma}.
	Finally, when $k=d \in \N$, we can follow the same line of argument. In particular, after \eqref{eq:ctil} (with $k=d$), we let $v_1=\sum_{i=1}^d{z_i}$ and $v_j=z_{j-1}$ for $2 \le j \le d$. Then arguing similarly as above, one arrives at
		\begin{align*}
		s^{\alpha'} c_{\alpha',\delta,\tau,s}^{(d)}(y)&\le \int_{-\log s}^{-\log (s|y|)}  \Big(-\log (s|y|) -
		v_1\Big)^{d-1}|v_1|^\delta \exp\Big\{-e^{-v_1} - (1+\alpha'+\tau)v_1\Big\} \diff v_1\\
		&\le \frac{(s|y|)^{\alpha'} e^{-s|y|/2}}{(s|y|)^{d\alpha'}}  \int_{s|y|}^{s} \Big(\log w-\log (s|y|)\Big)^{d-1} |\log w|^\delta w^{d\alpha'+\tau} e^{-w/2} \diff w.
	\end{align*}
An application of Jensen's inequality and \eqref{eq:Gamma} now imply that there exists a constant $C \in (0,\infty)$ such that
$$
c_{\alpha',\delta,\tau,s}^{(d)}(y) \le C\frac{(|y|)^{\alpha'}}{(s|y|)^{d\alpha'}} e^{-s|y|/2} \left[1+ \big|\log (s|y|)\big|^{d-1}\right]
$$
yielding the result.
\end{proof}

\begin{corollary}	\label{cor:c-bound}
	For $\alpha, s>0$, $\tau>-1$, $d \in \N$ and $\delta \ge 0$, the function
	\begin{equation}\label{eq:barc}
		c_{\alpha,\delta,\tau,s} (y):=s\int_{\XX} \mathds{1}_{x\succ y} \|x\|^\alpha
		e^{-s |x|} \big|\log (s |x|) \big|^\delta (s|x|)^\tau \diff x
	\end{equation}
satisfies
\begin{equation*}
	c_{\alpha,\delta,\tau,s}(y)\leq C
	\frac{\|y\|^{\alpha'}}{(s|y|)^{\alpha'}} e^{-s|y|/2}\Big[1+\big|\log(s|y|)\big|^{(d-2)^+ + \mathds{1}_{d=1}}\Big], \quad y \in \XX
\end{equation*}
for any $\alpha' \in (0,\alpha]$ for a constant $C \in (0,\infty)$ that depends on $\delta,\alpha, \alpha', \tau$ and $d$.
\end{corollary}
\begin{proof}
	As $\|x\| \le \sqrt{d}$, there exists a constant $C_1 \in (0,\infty)$ depending only on $\alpha$, $\alpha'$ and $d$ such that
	$$
	c_{\alpha,\delta,\tau,s}(y) \le C_1 s\int_\XX \mathds{1}_{x \succ y} \|x\|^{\alpha'} e^{-s|x|}\big|\log (s |x|) \big|^\delta (s |x|)^\tau \diff x= C_1 c_{\alpha',\delta,\tau,s}(y). 
	$$
	By equivalence of $L^2$ and $L^{\alpha'}$-norms, there exists a constant $C_2 \in (0,\infty)$ depending on $\alpha'$ and $d$ such that $C_2^{-1}\sum_{l=1}^d u_l^{\alpha'} \le  \|u\|^{\alpha'}  \le C_2 \sum_{l=1}^d u_l^{\alpha'}$ for $u \in \R^d$. Hence, using the second assertion in Lemma~\ref{lemma:c-bound} with $k=1$ in the second step, there exists $C' \in (0,\infty)$ such that
	\begin{align*}
	c_{\alpha',\delta,\tau,s}(y) &\le C_2 \sum_{l=1}^d s\int_\XX \mathds{1}_{x \succ y} x_l^{\alpha'} e^{-s|x|}\big|\log (s |x|) \big|^\delta (s |x|)^\tau \diff x\\
	& \le C_2C' \frac{\sum_{l=1}^d y_l^{\alpha'}}{(s|y|)^{\alpha'}} e^{-s|y|/2}\Big[1+\big|\log(s|y|)\big|^{(d-2)^+ + \mathds{1}_{d=1}}\Big] \\
	&\le C_2^2 C' \frac{\|y\|^{\alpha'}}{(s|y|)^{\alpha'}} e^{-s|y|/2}\Big[1+\big|\log(s|y|)\big|^{(d-2)^+ + \mathds{1}_{d=1}}\Big],
	\end{align*}
proving the result.
\end{proof}

Now we are ready to estimate $\Var{\mathscr{L}_0^{\alpha}}$ and prove Theorem~\ref{thm:MV}. In the following, for two points $x,y \in \XX$, we denote $x \wedge y:=(x_1 \wedge y_1, \dots, x_d \wedge y_d)$. First notice, letting $D$ denote the set of $(x,y) \in \XX^2$ such that $x$ and $y$ are
incomparable, i.e., $x \not \succ y$ and $y \not \succ x$, an application of the multivariate Mecke formula yields
\begin{align}\label{eq:Vspl}
	\Var{\mathscr{L}_0^{\alpha}}&= \E \sum_{x \in \sP_{s}^{min}} \|x\|^{2\alpha} - (\E \mathscr{L}_0^{\alpha})^2\nonumber\\
	&\qquad \qquad \qquad   + s^2 \iint_{D} \|x\|^{\alpha} \|y\|^{\alpha} \Prob{\{x,y\} \subseteq (\sP_s+\delta_x+\delta_y)^{min}} \diff x \diff y\nonumber\\
	& =s \int_\XX \|x\|^{2\alpha} e^{-s|x|} \diff x - I_{s0} + \sum_{k=1}^{d-1} \binom{d}{k} I_{sk},
\end{align}
where
\begin{equation*}
I_{s0}=2s^2 \int_{\XX^2} \mathds{1}_{y \prec x} \|x\|^{\alpha} \|y\|^{\alpha} e^{-s(|x|+|y|)} \diff x \diff y
\end{equation*}
and
\begin{equation*}
I_{sk}=s^2 \int_{\XX^2} \mathds{1}_{x^{(I_k)} \succ y^{(I_k)}, x^{(J_k)} \prec y^{(J_k)}}  \|x\|^{\alpha} \|y\|^{\alpha} e^{-s(|x|+|y|)} (e^{s|x \wedge y|} -1)  \diff x \diff y,
\end{equation*}
where we recall that for $1 \le k \le d-1$, $I_k=[k]$ and $J_k=[d] \setminus I_k$ and $x \wedge y=(y^{(I_k)},x^{(J_k)})$. By Corollary~\ref{cor:1}, for $d \ge 2$ and $s>1$ we have
\begin{equation}\label{eq:c}
	s \int_\XX \|x\|^{2\alpha} e^{-s|x|} \simeq \frac{d}{2\alpha (d-2)!} \log^{d-2} s.
\end{equation}
In the following two lemmas, we estimate $I_{s0}$ and $I_{sk}$ for $k \in [d-1]$ with $d \ge 2$. We will use the fact that by Lemma~\ref{lem:secord}, there exists $C \in (0,\infty)$ such that
\begin{align}\label{eq:split}
	&\|x\|^{\alpha} \|y\|^{\alpha} - \sum_{i=1}^d (x_i y_i)^{\alpha}=\left[\left(\|x\|^{\alpha} -  \sum_{i=1}^d x_i^\alpha\right) \|y\|^{\alpha} \right]\nonumber\\
	&\qquad \qquad + \left[ \left(\|y\|^{\alpha} - \sum_{i=1}^d y_i^\alpha \right) \sum_{i=1}^d x_i^\alpha \right] + \left[\left(\sum_{i=1}^d x_i^\alpha\right)\left(\sum_{i=1}^d y_i^\alpha\right) - \sum_{i=1}^d \sum_{i=1}^d (x_i y_i)^{\alpha}\right]\nonumber\\
	& \le C\left[ \sum_{i \not = j \in [d]} (x_i x_j)^{(1 \wedge \alpha)/2} + \sum_{i \not = j \in [d]} (y_i y_j)^{(1 \wedge \alpha)/2} + \sum_{i \not = j \in [d]} (x_i y_j)^{\alpha} \right].
\end{align}
\begin{lemma}\label{lem:Is0}
	For $d \ge 2$, $s>1$ and $\alpha>0$, 
	$$
	I_{s0}\simeq \left[\frac{d}{\alpha(d-2)!} \int_{\XX } \frac{b_1^\alpha }{(1+|b|)^{2}} \diff b\right] \log^{d-2} s.
	$$
\end{lemma}
\begin{proof}
	Using \eqref{eq:split} in the first step and that $s|x|e^{-s|x|/2} \le 1$, we have that there exists a constant $C \in (0,\infty)$ such that
		\begin{align*}
			&\Big| I_{s0}- 2d s^2 \int_{\XX^2} \mathds{1}_{y \prec x} (x_1 y_1)^{\alpha} e^{-s(|x|+|y|)}  \diff x \diff y \Big| \\
			&\le C s^2 \int_{\XX^2}\mathds{1}_{y \prec x} \left[(x_1 x_2)^{(1 \wedge \alpha)/2} + (y_1 y_2)^{(1 \wedge \alpha)/2} +(x_1 y_2)^{\alpha}\right] e^{-s(|x|+|y|)} \diff x \diff y\\
			& \le 3C s^2 \int_{\XX^2}\mathds{1}_{y \prec x} (x_1 x_2)^{(1 \wedge \alpha)/2} e^{-s(|x|+|y|)} \diff x \diff y\\
			&\le 3 C s \int_{\XX^2}(x_1 x_2)^{(1 \wedge \alpha)/2} e^{-s|x|/2} (s|x|e^{-s|x|/2}) \diff x =\mathcal{O}(\log^{d-3} s),
		\end{align*}
		where the last step is due to Theorem~\ref{thm:mean}. Using this in the first step, letting $b_i=y_i/x_i$, $i \in [d]$ in the second, substituting $s^{1/d}x=u$ in the third, and then following the same series of substitutions as in Theorem~\ref{thm:mean}, we obtain
		\begin{align}\label{eq:Is0}
			I_{s0} &\simeq 2d s^2 \int_{\XX} x_1^{\alpha} e^{-s|x|} \int_{\XX } \mathds{1}_{y \prec x} y_1^{\alpha} e^{-s|y|}  \diff y \diff x\nonumber\\
			&=2d s \int_{\XX } b_1^\alpha \int_{\XX}  s|x| x_1^{2\alpha} e^{-(1+|b|)s|x|}  \diff x \diff b\nonumber\\
			&=2d s^{-2\alpha/d}\int_{\XX } b_1^\alpha \int_{[0, s^{1/d}]^d} u_1^{2\alpha} |u| e^{-(1+|b|)|u|} \diff u \diff b\nonumber\\
			&=\frac{2d s^{-2\alpha}}{(d-2)!} \int_{\XX }b_1^\alpha  \int_0^s \int_{w_1}^{s} \left(\log \bar w_2 - \log w_1 \right)^{d-2}w_1  e^{-(1+|b|)w_1} \bar w_2^{2\alpha-1} \diff \bar w_2 \diff w_1 \diff b\nonumber\\
			&=\frac{2d s^{-2\alpha}}{(d-2)!} \int_{\XX } b_1^\alpha \sum_{i=0}^{d-2} \binom{d-2}{i} (-1)^i  \int_0^s \int_{0}^{\bar w_2} (\log w_1)^i w_1  e^{-(1+|b|)w_1} (\log \bar w_2)^{d-2-i}  \bar w_2^{2\alpha-1}  \diff w_1 \diff \bar w_2  \diff b\nonumber\\
			&\simeq \frac{2 d s^{-2\alpha}}{(d-2)!} \int_{\XX } b_1^\alpha \int_{\sqrt{s}}^s \int_{0}^{\bar w_2} w_1  e^{-(1+|b|)w_1} (\log \bar w_2)^{d-2} \bar w_2^{2\alpha-1}  \diff w_1 \diff \bar w_2  \diff b,
		\end{align}
	where in the final step, we used the fact that by \eqref{eq:loges} and \eqref{eq:Gamma}, there exists $C \in (0,\infty)$ such that
	$$
	\int_0^{\sqrt{s}} \int_{0}^{\bar w_2} w_1  e^{-(1+|b|)w_1} (\log \bar w_2)^{d-2} \bar w_2^{2\alpha-1}  \diff w_1 \diff \bar w_2 \le C \int_0^{\sqrt{s}} (\log \bar w_2)^{d-2} \bar w_2^{2\alpha-1} \diff \bar w_2 =\mathcal{O}(s^\alpha \log^{d-2}s).
	$$
		Finally, as $\int_{0}^{\bar w_2}  w_1  e^{-(1+|b|)w_1}\diff w_1=(1+|b|)^{-2}+\mathcal{O}(e^{-\sqrt{s}/2})$ for $\bar w_2 \ge \sqrt{s}$, we have from \eqref{eq:Is0},
		\begin{align*}
			I_{s0}
			&\simeq\frac{2d s^{-2\alpha}}{(d-2)!} \int_{\XX } \frac{b_1^\alpha }{(1+|b|)^{2}} \diff b \int_{\sqrt{s}}^{s} (\log \bar w_2)^{d-2} \bar w_2^{2\alpha-1} \diff \bar w_2\simeq \left[\frac{d}{\alpha(d-2)!} \int_{\XX } \frac{b_1^\alpha }{(1+|b|)^{2}} \diff b\right] \log^{d-2} s,
		\end{align*}
		where the final step is due to \eqref{eq:loges}. 
\end{proof}

\begin{lemma} For $d \ge 2$, $s>1$, $\alpha>0$,  and $k \in [d-1]$,
	$$
	I_{sk} \simeq \left[\frac{1}{2\alpha(d-2)!} \int_{\XX } \Big(k b_1^{\alpha} + (d-k)b_{d}^{\alpha}\Big) \left(\frac{1}{(|b^{(I_k)}| + |b^{(J_k)}|-|b|)^{2}}  - \frac{1}{(|b^{(I_k)}| + |b^{(J_k)}|)^{2}}\right) \diff b \right] \log^{d-2} s.
	$$
\end{lemma}
\begin{proof}
	Fix $k \in [d-1]$ and denote $r:=(y^{(I_k)},x^{(J_k)})$ and $R:=(x^{(I_k)},y^{(J_k)})$. 
	Fix $t>0$ and $i\not=j \in [d]$. When $\{i,j\} \subseteq I_k$ (which also implies $d \ge 3$), using the inequality $e^x - 1 \le xe^{x}$ for $x \ge 0$ in the first step and the fact that $|x|+|y|-|r| \ge (|x|+|y|)/2$ in the second, there exists some constant $C \in (0,\infty)$ such that
	\begin{align*}
		&s^2 \int_{\XX^2} \mathds{1}_{r \prec R} (R_i R_j)^{t}  e^{-s(|x|+|y|)} (e^{s|r|}-1) \diff r \diff R
		\le s^2 \int_{\XX^2} \mathds{1}_{r \prec R} (R_iR_j)^{t}  e^{-s(|x|+|y|)} s|r| e^{s|r|} \diff r \diff R \nonumber\\
		&\qquad \le s \int_\XX \left(s|r^{(J_k)}| \int_{[0,1]^{k}} \mathds{1}_{R^{(I_k)} \succ r^{(I_k)}} (R_iR_j)^{t} e^{-s|R^{(I_k)}||r^{(J_k)}|/2} \diff R^{(I_k)}\right)  \nonumber\\
		& \qquad\qquad\qquad\qquad\qquad \times \left(s|r^{(I_k)}| \int_{[0,1]^{d-k}} \mathds{1}_{R^{(J_k)} \succ r^{(J_k)}} e^{-s|r^{(I_k)}||R^{(J_k)}|/2} \diff R^{(J_k)}\right) \diff r\nonumber\\
		& \qquad \le C s \int_\XX \frac{(r_ir_j)^{t'}}{(s|r|)^{2t'}} e^{-s|r|/2} \Big[1+\big|\log(s|r|/2)\big|^{(k-3)^+ +\mathds{1}_{k=2}}\Big] \Big[1+\big|\log(s|r|/2)\big|^{d-k-1}\Big]\diff r
	\end{align*}
	for some $t' \in [0,t]$ with $2t'<1$, where we have used the second assertion in Lemma~\ref{lemma:c-bound} for the last step. Hence, by Theorem~\ref{thm:mean}, we obtain
	\begin{equation}\label{eq:secik}
		s^2 \int_{\XX^2} \mathds{1}_{r \prec R} (R_i R_j)^{t}  e^{-s(|x|+|y|)} (e^{s|r|}-1) \diff r \diff R \le \mathcal{O}(\log^{d-3} s).
	\end{equation}
	By symmetry, this also holds when $\{i,j\} \subseteq J_k$. Finally, when $i \in I_k$ and $j \in J_k$ (and symmetrically, when $i \in J_k$ and $j \in I_k$), arguing similarly, by Lemma~\ref{lemma:c-bound} and Theorem~\ref{thm:mean}, we have that there exists $C \in (0,\infty)$ and $t' \in [0,t]$ with $2t'<1$ such that
	\begin{align}\label{eq:secik'}
		&s^2 \int_{\XX^2} \mathds{1}_{r \prec R} (R_i R_j)^{t}  e^{-s(|x|+|y|)} (e^{s|r|}-1) \diff r \diff R\nonumber\\
		&\le s \int_\XX \left(s|r^{(J_k)}| \int_{[0,1]^{k}} \mathds{1}_{R^{(I_k)} \succ r^{(I_k)}} R_i^{t} e^{-s|R^{(I_k)}||r^{(J_k)}|/2} \diff R^{(I_k)}\right)  \nonumber\\
		& \qquad\qquad\qquad\qquad\qquad \times \left(s|r^{(I_k)}| \int_{[0,1]^{d-k}} \mathds{1}_{R^{(J_k)} \succ r^{(J_k)}} R_j^t e^{-s|r^{(I_k)}||R^{(J_k)}|/2} \diff R^{(J_k)}\right) \diff r\nonumber\\
		& \le C s \int_\XX \frac{(r_ir_j)^{t'}}{(s|r|)^{2t'}} e^{-s|r|/2} \Big[1+\big|\log(s|r|/2)\big|^{(k-2)^+}\Big] \Big[1+\big|\log(s|r|/2)\big|^{(d-k-2)^+}\Big]\diff r = \mathcal{O}(\log^{d-3} s).
	\end{align}
	Putting together \eqref{eq:split}, \eqref{eq:secik} and \eqref{eq:secik'}, we obtain
	\begin{align*}
		I_{sk} & \simeq 
		\sum_{j=1}^d s^2 \int_{\XX^2}\mathds{1}_{r \prec R} (r_jR_j)^{\alpha} e^{-s(|x|+|y|)} (e^{s|r|}-1) \diff r \diff R.
	\end{align*}
	Writing $b_j=r_j/R_j$ for $j \in [d]$ in the first step and letting $s^{1/d} R=u$ in the second, arguing as in in case of $I_{s0}$ in Lemma~\ref{lem:Is0}, we obtain,
	\begin{align}\label{eq:Is1}
		I_{sk}
		&\simeq \sum_{j=1}^d s^2 \int_{\XX^2} |R| R_j^{2\alpha} \int_{\XX}b_j^{\alpha} e^{-(|b^{(I_k)}| + |b^{(J_k)}|)s|R|} (e^{s|R||b|}-1) \diff b \diff R\nonumber\\
		& = s^{-2\alpha/d} \int_\XX \sum_{j=1}^d b_j^{\alpha}  \int_{[0,s^{1/d}]^d} u_1^{2\alpha}|u| \left[e^{-(|b^{(I_k)}| + |b^{(J_k)}|-|b|)|u|} - e^{-(|b^{(I_k)}| + |b^{(J_k)}|)|u|} \right]\diff u \diff b\nonumber\\
		&=\frac{s^{-2\alpha}}{(d-2)!} \int_{\XX } \sum_{j=1}^d b_j^{\alpha} \sum_{i=0}^{d-2} \binom{d-2}{i} (-1)^i  \int_0^s \int_{0}^{\bar w_2} (\log w_1)^i w_1 \nonumber\\
		& \qquad \qquad \times \left[e^{-(|b^{(I_k)}| + |b^{(J_k)}|-|b|)w_1} - e^{-(|b^{(I_k)}| + |b^{(J_k)}|)w_1} \right] (\log \bar w_2)^{d-2-i}  \bar w_2^{2\alpha-1}  \diff w_1 \diff \bar w_2  \diff b.
	\end{align}
	Let $j \in [d]$ and $i \in [d-2]$. Using \eqref{eq:loges} in the first step and the inequality $e^{-x}-e^{-y} \le (y-x)e^{-x}$ for $y \ge x \ge 0$ in the second, we have
	\begin{align}\label{eq:Is2}
		&\int_{\XX } b_j^\alpha \int_0^s \int_{0}^{\bar w_2} |\log w_1|^i w_1  \left[e^{-(|b^{(I_k)}| + |b^{(J_k)}|-|b|)w_1} - e^{-(|b^{(I_k)}| + |b^{(J_k)}|)w_1} \right] (\log \bar w_2)^{d-2-i}  \bar w_2^{2\alpha-1}  \diff w_1 \diff \bar w_2  \diff b\nonumber\\
		& \le C s^{2\alpha} \log^{d-2-i}s  \int_{\XX } \int_{0}^{\infty} |\log w_1|^i w_1  \left[e^{-(|b^{(I_k)}| + |b^{(J_k)}|-|b|)w_1} - e^{-(|b^{(I_k)}| + |b^{(J_k)}|)w_1} \right] \diff w_1  \diff b\nonumber\\
		& \le C s^{2\alpha} \log^{d-2-i}s  \int_{\XX } |b| \int_0^\infty |\log w_1|^iw_1^2  e^{-(|b^{(I_k)}| + |b^{(J_k)}|-|b|)w_1} \diff w_1 \diff b
		\end{align}
for some constant $C>0$.
Let $C(i)$ be such that $(\log w)^i \le C(i) w^{1/4}$ for $w \ge 1$. Using \eqref{eq:Gamma} in the second step and that $(|b^{(I_k)}| + |b^{(J_k)}|-|b|)^{2} \ge |b|$ in the third, we have
	\begin{align}\label{eq:arg}
	&\int_{\XX } |b| \int_0^\infty |\log w_1|^i w_1^2  e^{-(|b^{(I_k)}| + |b^{(J_k)}|-|b|)w_1} \diff w_1 \diff b\nonumber\\
	&\le \int_0^1 |\log w_1|^i \diff w_1+C(i) \int_{\XX } |b| \int_1^\infty w_1^{9/4}  e^{-(|b^{(I_k)}| + |b^{(J_k)}|-|b|)w_1} \diff w_1 \diff b\nonumber\\
	& \le \int_0^1 |\log w_1|^i \diff w_1 + \Gamma (13/4) C(i) \int_{\XX }  \frac{ |b|}{(|b^{(I_k)}| + |b^{(J_k)}|-|b|)^{13/4}}  \diff b\nonumber\\
	& \le \int_0^1 |\log w_1|^i \diff w_1 + \Gamma (13/4) C(i) \int_{\XX } \frac{1}{(|b^{(I_k)}| + |b^{(J_k)}|-|b|)^{5/4}}  \diff b<\infty,
\end{align}
where the final step follows upon noticing $(|b^{(I_k)}| + |b^{(J_k)}|)/2 \ge \sqrt{|b|} \ge |b|$ so that
\begin{align*}
	\int_{\XX } \frac{1}{(|b^{(I_k)}| + |b^{(J_k)}|-|b|)^{5/4}}  \diff b \le \int_{\XX } \frac{1}{|b|^{5/8}}  \diff b <\infty.
\end{align*}
Also by \eqref{eq:loges},
	\begin{align*}
		&\int_{\XX }  \int_0^{s^{\alpha/(2\alpha+2)}} \int_{0}^{\bar w_2} w_1 \left[e^{-(|b^{(I_k)}| + |b^{(J_k)}|-|b|)w_1} - e^{-(|b^{(I_k)}| + |b^{(J_k)}|)w_1} \right](\log \bar w_2)^{d-2}  \bar w_2^{2\alpha-1}  \diff w_1 \diff \bar w_2  \diff b\nonumber\\
		& \le \int_0^{s^{\alpha/(2\alpha+2)}}  (\log \bar w_2)^{d-2}  \bar w_2^{2\alpha+1} \diff \bar w_2  =\mathcal{O}(s^\alpha \log^{d-2}s).
	\end{align*}
	Combining the above two estimates with \eqref{eq:Is1} and \eqref{eq:Is2} yields
	\begin{align}\label{eq:2.23}
		&I_{sk} \simeq \frac{s^{-2\alpha}}{(d-2)!} \int_{\XX } \sum_{j=1}^d b_j^{\alpha} \int_{s^{\alpha/(2\alpha+2)}}^s \int_{0}^{\bar w_2} w_1 \nonumber\\
		& \qquad \qquad \qquad \times \left[e^{-(|b^{(I_k)}| + |b^{(J_k)}|-|b|)w_1} - e^{-(|b^{(I_k)}| + |b^{(J_k)}|)w_1} \right] (\log \bar w_2)^{d-2}  \bar w_2^{2\alpha-1}  \diff w_1 \diff \bar w_2  \diff b.
	\end{align}
Since $\int_0^\infty xe^{-\beta x} \diff x=\beta^{-2}$ for $\beta>0$, for $\bar w_2 \ge s^{\alpha/(2\alpha+2)}$, 
\begin{align*}
	&\Bigg|\int_{0}^{\bar w_2} w_1 \Big[e^{-(|b^{(I_k)}| + |b^{(J_k)}|-|b|)w_1}- e^{-(|b^{(I_k)}| + |b^{(J_k)}|)w_1} \Big]  \diff w_1 \\
	& \qquad \qquad\qquad \qquad\qquad \qquad-\left(\frac{1}{(|b^{(I_k)}| + |b^{(J_k)}|-|b|)^{2}}  - \frac{1}{(|b^{(I_k)}| + |b^{(J_k)}|)^{2}}\right)\Bigg|\\
	& \le \int_{s^{\alpha/(2\alpha+2)}}^{\infty} w_1 \left[e^{-(|b^{(I_k)}| + |b^{(J_k)}|-|b|)w_1} - e^{-(|b^{(I_k)}| + |b^{(J_k)}|)w_1} \right]  \diff w_1\\
	& \le s^{-\frac{\alpha}{8(\alpha+1)}} |b| \int_{0}^{\infty} w_1^{9/4} e^{-(|b^{(I_k)}| + |b^{(J_k)}|-|b|)w_1}  \diff w_1.
\end{align*}
So by \eqref{eq:2.23},
\begin{align*}
	&\Big|I_{sk} - \frac{s^{-2\alpha}}{(d-2)!} \int_{\XX } \sum_{j=1}^d b_j^{\alpha} \int_{s^{\alpha/(2\alpha+2)}}^s  \nonumber\\
	& \qquad \qquad \qquad \times \left(\frac{1}{(|b^{(I_k)}| + |b^{(J_k)}|-|b|)^{2}}  - \frac{1}{(|b^{(I_k)}| + |b^{(J_k)}|)^{2}}\right) (\log \bar w_2)^{d-2}  \bar w_2^{2\alpha-1}  \diff \bar w_2  \diff b\Big|\\
	&\le s^{-\frac{\alpha}{8(\alpha+1)}}  \frac{s^{-2\alpha}d}{(d-2)!} \int_{\XX } |b| \int_{s^{\alpha/(2\alpha+2)}}^s \int_0^\infty w_1^{9/4} e^{-(|b^{(I_k)}| + |b^{(J_k)}|-|b|)w_1}  (\log \bar w_2)^{d-2}  \bar w_2^{2\alpha-1} \diff w_1 \diff \bar w_2  \diff b\\
	&=\mathcal{O}(s^{-\frac{\alpha}{8(\alpha+1)}} \log^{d-2} s)  \int_{\XX } |b| \int_0^\infty w_1^{9/4} e^{-(|b^{(I_k)}| + |b^{(J_k)}|-|b|)w_1}  \diff w_1  \diff b=\mathcal{O}(s^{-\frac{\alpha}{8(\alpha+1)}}  \log^{d-2} s),
\end{align*}
where the final step is argued as in \eqref{eq:arg}. Hence,
\begin{align*}
	I_{sk}&\simeq \frac{s^{-2\alpha}}{(d-2)!} \int_{\XX } \sum_{j=1}^d b_j^{\alpha} \left(\frac{1}{(|b^{(I_k)}| + |b^{(J_k)}|-|b|)^{2}}  - \frac{1}{(|b^{(I_k)}| + |b^{(J_k)}|)^{2}}\right) \diff b \nonumber\\ 
	& \qquad \qquad \qquad \qquad \qquad\times \int_{s^{\alpha/(2\alpha+2)}}^s (\log \bar w_2)^{d-2}  \bar w_2^{2\alpha-1}  \diff \bar w_2\\
	& \simeq \left[\frac{1}{2\alpha(d-2)!} \int_{\XX } \sum_{j=1}^d b_j^{\alpha} \left(\frac{1}{(|b^{(I_k)}| + |b^{(J_k)}|-|b|)^{2}}  - \frac{1}{(|b^{(I_k)}| + |b^{(J_k)}|)^{2}}\right)\diff b \right] \log^{d-2} s,
\end{align*}
where the final step is due to \eqref{eq:loges}. The desired conclusion follows by symmetry.
\end{proof}

Collecting the estimates from the above two lemmas and combining with \eqref{eq:c} and \eqref{eq:Vspl}, we obtain that for $d \ge 2$ and $s >1$,
\begin{equation}\label{eq:var}
	\Var{\mathscr{L}_0^{\alpha}} \simeq \frac{1}{2\alpha(d-2)!} w(d,\alpha) \log^{d-2} s,
\end{equation}
where $w(d,\alpha)$ is defined at \eqref{eq:w}. As the last ingredient in the proof of Theorem~\ref{thm:MV}, we now show that $w(d,\alpha)$ is finite and positive.

\begin{lemma}\label{lem:wa} For $d \ge 2$, the function $w$ given by \eqref{eq:w}
	satisfies 
	$$
	0<\inf_{\alpha>0} w(d,\alpha)\le \sup_{\alpha>0} w(d,\alpha)<\infty.
	$$
\end{lemma}

\begin{proof} First, using mean value theorem and arguing as in \eqref{eq:arg}, for any $\alpha>0$ we have
	\begin{align*}
		&\int_{\XX } b_1^{\alpha} \left(\frac{1}{(|b^{(I_k)}| + |b^{(J_k)}|-|b|)^{2}}  - \frac{1}{(|b^{(I_k)}| + |b^{(J_k)}|)^{2}}\right)\diff b \le 2 \int_{\XX }\frac{|b|}{(|b^{(I_k)}| + |b^{(J_k)}|-|b|)^{3}}\diff b <\infty,
	\end{align*}
	implying $\sup_{\alpha>0} w(d,\alpha)<\infty.$
	
	Next, notice that for all $b \in \XX$ and $k \in [d-1]$, we have $|b^{(I_k)}| + |b^{(J_k)}|-|b| \le 1$. Hence, 
	\begin{align*}
		2\sum_{k=1}^{d-1} k\binom{d}{k}\int_{\XX } & b_1^{\alpha} \left(\frac{1}{(|b^{(I_k)}| + |b^{(J_k)}|-|b|)^{2}}  - \frac{1}{(|b^{(I_k)}| + |b^{(J_k)}|)^{2}}\right)\diff b\\
		&\ge (2^d-2) d \int_{\XX }b_1^{\alpha} \left(1  - \frac{1}{(1+|b|)^{2}}\right)\diff b.
	\end{align*}
By substituting $b_1'=b_1^{\alpha +1}$, notice 
	$$
	\int_{\XX} \frac{b_1^{\alpha}}{(1+|b|)^{2}} \diff b \le \frac{1}{\alpha+1} \int_{\XX} \frac{1}{(1+b_1' b_2 \cdots b_d)^{2}} \diff (b_1',b_2,\dots,b_d)=	\frac{1}{\alpha+1}\int_{\XX} \frac{1}{(1+|b|)^{2}} \diff b.
	$$
	Hence, for any $\alpha>0$ we obtain,
	\begin{multline}\label{eq:sim}
		d^{-1}w(d,\alpha) \ge \left(1 + \frac{2^d-2}{\alpha+1}\right) - \frac{2^d}{\alpha+1} \int_{\XX} \frac{1}{(1+|b|)^2} \diff b \\
= \frac{2^d}{\alpha+1} \left[ \frac{\alpha +1+2^d-2}{2^d} - \int_{\XX} \frac{1}{(1+|b|)^2} \diff b\right]
 \ge \frac{2^d}{\alpha+1} \left[ \frac{2^d-1}{2^d} - \int_{\XX} \frac{1}{(1+|b|)^2} \diff b\right].
	\end{multline}
	Next, we claim that
	\begin{equation}\label{eq:sqint}
		\int_{[0,1]^d} \frac{1}{(1+|b|)^2} \diff b = \left\{
		\begin{array}{cc}
			\frac{1}{2} & d=1 \\
			\log 2 & d=2\\
			\frac{2^{d-2} -1}{2^{d-2}} \zeta(d-1) & d \ge 3,
		\end{array}
		\right.
	\end{equation}
	where $\zeta$ is the Riemann zeta function. Indeed, the statement is trivial for $d=1$. For $b=(b_1,\dots,b_d) \in \XX$, recall that $b^{(I_{d-1})}:=(b_1,\dots,b_{d-1})$. For $d \ge 2$, notice that
	\begin{equation}\label{eq:6}
		\int_{[0,1]^d} \frac{1}{(1+|b|)^2} \diff b=\int_{[0,1]^{d-1}} \frac{1}{|b^{(I_{d-1})}|} \int_0^{|b^{(I_{d-1})}|} \frac{1}{(1+t)^2} dt \diff b^{(I_{d-1})}=\int_{[0,1]^{d-1}} \frac{1}{1+|b^{(I_{d-1})}|} \diff b^{(I_{d-1})}.
	\end{equation}
	Thus, we have $\int_{[0,1]^2} \frac{1}{(1+|b|)^2} \diff b=\log 2$.
	On the other hand, for $d \ge 3$, substituting $b_i^2=c_i$ for $i \in [d-1]$ in the final step, we obtain 
	\begin{align*}
		\int_{[0,1]^{d-1}} &\left[\frac{1}{1-|b^{(I_{d-1})}|} -  \frac{1}{1+|b^{(I_{d-1})}|} \right]\diff b^{(I_{d-1})}\\
		&= \int_{[0,1]^{d-1}} \frac{2|b^{(I_{d-1})}|}{1-|b^{(I_{d-1})}|^2} \diff b^{(I_{d-1})}=\frac{1}{2^{d-2}}	\int_{[0,1]^{d-1}} \frac{1}{1-|c|}  \diff c.
	\end{align*}
	Thus, by \eqref{eq:6},
	$$
	\int_{[0,1]^d} \frac{1}{(1+|b|)^2} \diff b= \frac{2^{d-2} -1}{2^{d-2}} \int_{[0,1]^{d-1}} \frac{1}{1-|c|}  \diff c= \frac{2^{d-2} -1}{2^{d-2}} \zeta(d-1),
	$$
	where the final step is obtained by writing $1/(1-|c|)$ as a geometric series. This proves \eqref{eq:sqint}. Finally, by using the approximation that for any $i \ge 2$,
	$$
	\zeta(i) \le \sum_{j=1}^{3}\frac{1}{j^i}+ \int_{3}^\infty \frac{1}{s^i} \diff s=\sum_{j=1}^{3}\frac{1}{j^i}+ \frac{1}{(i-1)3^{i-1}},
	$$
	it is not hard to show that for $d \ge 4$,
	$$
	\frac{2^{d-2} -1}{2^{d-2}} \zeta(d-1) <  \frac{2^d-1}{2^d}.
	$$
	Plugging \eqref{eq:sqint} in \eqref{eq:sim}, using the above inequality for $d \ge 4$ and checking the case for $d=2,3$ by hand yields the desired lower bound.
\end{proof}

\begin{proof}[Proof of Theorem~\ref{thm:MV}.]
	As mentioned after Corollary~\ref{cor:1}, assertion (a) follows from the corollary upon taking $\beta=1$ and $\tau=0$. Assertion (b) in Theorem~\ref{thm:MV} follows directly from \eqref{eq:var} and Lemma~\ref{lem:wa}. 
\end{proof}

\section{Proof of Theorem~\ref{thm:Pareto}}\label{sec:Pareto}
Recall, $\QQ$ is the Lebesgue
measure on $\XX:=[0,1]^d$ with $d \ge 3$, and $\sP_s$ is a Poisson
process on $\XX$ with intensity measure $s\QQ$ for $s \ge 1$. Notice that the
functional $\mathscr{L}_0^{\alpha}$ from \eqref{eq:ParetoPoints} is expressible as in
\eqref{eq:hs} with
\begin{equation}
\label{eq:xi}
\xi_s(x,\M):=\|x\|^\alpha \mathds{1}_{x \in \mu^{min}},\quad \; x\in\M, \; \M \in \Nb.  
\end{equation}
That $(\xi_s)_{s \ge 1}$ satisfies
condition (A0) is straightforward to see. Indeed, notice that for $\M_1, \M_2 \in \Nb$ with $\M_1 \le \M_2$ and $x \in \M_1$, the equality $\xi_s(x,\M_1)=\xi_s(x,\M_2)$ implies that $x$ is either minimal in both $\M_1$ and $\M_2$ or it is not minimal in both. In either case, for $\M \in \Nb$ with $\M_1 \le \M \le \M_2$, it is easy to check that $\mathds{1}_{x \in \mu^{min}}=\mathds{1}_{x \in \mu_1^{min}}=\mathds{1}_{x \in \mu_2^{min}}$, which readily implies (A0). In the following, we show that conditions
(A1), (A2) also hold true, so that we can apply Theorem~\ref{thm:KolBd} to prove Theorem~\ref{thm:Pareto}.

Given a counting measure $\M\in\Nb$ with $x \in \mu$, let the stabilization
region be
\begin{equation*}
R_s(x,\M):= 
\begin{cases}
[0,x] & \mbox{if $\M([0,x]\setminus \{x\})=0$},\\
\emptyset & \mbox{otherwise}.
\end{cases} 
\end{equation*}
It is easy to see (see also \cite[Section~3]{BM21})
that (A1) is satisfied with $\xi_s$ defined at \eqref{eq:xi}. Letting $M_{s}(x) = \|x\|^\alpha$ for all $p \in (0,1]$, we have that (A2) holds trivially for such $p$ and $s \ge 1$. For definiteness, we take $p=1$. Thus, $\widetilde{M}_{s}(x)=\max\{\|x\|^{2\alpha},\|x\|^{4\alpha}\}$.

Inequality \eqref{eq:Rs} is satisfied by $\xi_s$ with $r_{s}(x,y):=s|x|$
if $y\preceq x$ and $r_{s}(x,y):=\infty$ if $y \not \preceq x$. 

Recall the function $c_{\beta,s}$ and $c_{\alpha,\delta,\tau,s}$ from \eqref{eq:cdef} and \eqref{eq:barc}, respectively, and note that $g_{s}(y)$ and $h_s(y)$ from
\eqref{eq:g} are equal to $c_{\lambda,s}(y)$ and $\lambda^{-1} c_{4+p/2,0,0,\lambda s} (y)$, respectively, with $\lambda=p/(40+10p)$. For brevity of notation, we will simply write $c(y)$ and $\bar c(y)$ for $c_{\lambda,s}(y)$ and $c_{4+p/2,0,0,\lambda s}(y)$, respectively. In the rest of the section,
$x^{(1)}\vee\dots\vee x^{(n)}$ stands for the coordinatewise maximum of
$x^{(1)},\dots,x^{(n)} \in \XX$, while $x^{(1)}\wedge\dots\wedge x^{(n)}$ denotes
the coordinatewise minimum. For $x,y\in \XX$, notice that
$\{x,y\}\subseteq R_{s}(z,\sP_{s}+\delta_z)$ if and only if
$z\succ (x \vee y)$ and $[0,z] \setminus \{z\}$ has no points of $\sP_s$. Thus, the function $q_{s}$ from \eqref{eq:g2s} is given by
\begin{displaymath}
q_{s}(x,y):=s \int_{\XX} \P\big\{\{x,y\} \subseteq R_{s}(z,\sP_{s}
+\delta_z)\big\} \diff z 
= s \int_\XX \mathds{1}_{z\succ (x\vee y)} e^{-s|z|}\diff z
=c_{1,s}(x\vee y).
\end{displaymath}

Before proceeding to estimate the bound in Theorem~\ref{thm:KolBd}, we
need to prove a few lemmas. Recall the function $c_{\delta,\tau,s}$ defined at \eqref{eq:cal}. 
The following lemma is a slight modification (with an additional logatithmic and polynomial term) of \cite[Lemma~3.2]{BM21}, and follows from Lemma~\ref{lemma:c-bound} above mimicing the arguments in \cite[Lemma~3.2]{BM21}.
\begin{lemma}[Lemma~3.2, \cite{BM21}]\label{lem:intbd}
	For all $d,i \in \N$, $s>1$, $\delta \ge 0$ and $\tau>-1$,
	\begin{equation*}
	s\int_{\XX} c_{\delta,\tau,s}(y)^i \diff y=
	\mathcal{O}(\log^{d-1} s).
	\end{equation*}
\end{lemma}

We prove a version of the above result for $c_{\alpha,\delta,\tau,s}$ with non-trivial $\alpha>0$. The crucial difference here is that the addition of a norm in the integrand decreases the order of the integral by a logarithmic factor.
\begin{lemma}\label{lem:intbd'}
	Let $d \ge 2$. For all $i \in \N$, $s>1$, $\alpha>0, \delta \ge 0$ and $\tau>-1$,
	\begin{equation*}
	s\int_{\XX}  c_{\alpha,\delta,\tau,s}(y)^i \diff y=
	\mathcal{O}(\log^{d-2} s).
	\end{equation*}
\end{lemma}
\begin{proof}
	For $d \ge 2$ and $i \in \N$, taking $\alpha' \in (0,\alpha]$ such that $i\alpha'<1$, by Corollary~\ref{cor:c-bound} and Jensen's inequality, we have
	\begin{equation*}
	s\int_{\XX} c_{\alpha,\delta,\tau,s}(y)^i \diff y
	\le 2^{i-1} C'^i \left[s \int_{\XX}  \frac{\|y\|^{i\alpha'}}{(s|y|)^{i\alpha'}}  e^{- is|y|/2} \diff y
	+  s \int_{\XX}  \frac{\|y\|^{i\alpha'}}{(s|y|)^{i\alpha'}} e^{- is|y|/2} \big|\log(s|y|)\big|^{i(d-2)}
	\diff y\right],
	\end{equation*}
	with $C'$ as in Corollary~\ref{cor:c-bound}. An application of \eqref{eq:3.1} yields the result.
\end{proof}

Next we provide a key technical lemma needed to prove Theorem~\ref{thm:Pareto}. Before stating it, we note the following inequality. For any $s>0$, $\delta \ge 0$ and $\tau>-1$, following the computation for mean in Theorem~\ref{thm:mean} by writing $s'=s|y|$ and substituting $w_i=x_i/y_i$ for the first step, then $u_i=s'^{1/d} w_i$ followed by $z_i=-\log u_i$, $i\in [d]$ and finally $v=e^{-\sum_{i=1}^d z_i}$ to obtain the second equality, we have
\begin{multline}\label{eq:logint}
s \int_\XX \mathds{1}_{x \prec y} \big|\log (s|x|)\big|^\delta (s|x|)^\tau \diff x
=s' \int_{\XX} \big|\log (s'|w|)\big|^\delta (s'|w|)^\tau \diff w  \\
=\frac{1}{(d-1)!} \int_0^{s'} \left(\log s' - \log v \right)^{d-1} | \log v|^\delta v^\tau\diff v \le C (s|y|)^{1+\tau} \left(1+\big|(\log (s|y|)\big|^{d-1+\lceil\delta\rceil}\right)
\end{multline}
for some constant $C \in (0,\infty)$ depending on $d$, $\delta$ and $\tau$, where in the last step we have used Jensen's inequality and an elementary
inequality, saying that, for $l >0$ and
$a>0$, there exists a constant $b_{l,\tau} \in (0,\infty)$ depending on $l$ and $\tau$ such
that
\begin{displaymath}
\int_0^{a} |\log w|^l w^\tau\diff w = \frac{1}{(1+\tau)^{l+1}} \int_0^{a^{1+\tau}} |\log b|^l \diff b \le b_{l,\tau} a^{1+\tau}
\left[1+\sum_{i=1}^{\lceil l \rceil} |\log a|^i\right].
\end{displaymath}

\begin{lemma}\label{lem:maxlog} For $d \ge 2$, $s>1$, $i \in \N$, $\alpha>0$, $\tau, \tau'>-1$ and $\delta,\delta' \ge 0$,
	\begin{gather*}
	%\label{eq:c2'}
	s\int_{\XX}  \left(s \int_{\XX} \|x \vee y\|^\alpha  \big| \log (s|x \vee y|)\big|^\delta \big| \log (s|x|)\big|^{\delta'} (s|x|)^\tau e^{-s|x\vee
		y|}  \diff x\right)^i  \diff y=
	\mathcal{O}(\log^{d-2} s),\\
	%\label{eq:c3'}
	s \int_{\XX} \left(s\int_{\XX} \|x\|^\alpha (s|x|)^\tau c_{\delta,\tau',s}(x\vee y) \diff x
	\right)^i \diff y
	= \mathcal{O}(\log^{d-2} s). 
	\end{gather*}
\end{lemma}

\begin{proof}
	We start by proving the first assertion. 
	Recall, for $x \in \XX$ and $I \subseteq [d]$, we write $x^{(I)}$ for
	the subvector $(x_i)_{i \in I}$. We can always write $x\vee y=(x^{(I)},y^{(J)})$ for some $I\subseteq [d]$ with
	$J:=[d]\setminus I$. 
	By Jensen's inequality, we have
	\begin{align}\label{eq:maxsplit}
	&2^{-(i-1)d}s\int_{\XX}  \Bigg(s \int_{\XX} \|x \vee y\|^\alpha \big| \log (s|x \vee y|)\big|^\delta  \big| \log (s|x|)\big|^{\delta'} (s|x|)^\tau e^{- s|x\vee
		y|}  \diff x\Bigg)^i \diff y\nonumber\\
	&\le 
	\sum_{I \subseteq [d]} s\int_{\XX}
	\Bigg(s \int_\XX \mathds{1}_{x^{(I)} \succ y^{(I)}, x^{(J)} \prec y^{(J)}}\|x \vee y\|^\alpha  \big| \log (s|x^{(I)}| |y^{(J)}|)\big|^\delta \nonumber\\
	& \qquad \qquad \qquad\qquad \qquad\qquad \qquad\times \big| \log (s|x|)\big|^{\delta'} (s|x|)^\tau e^{- s|x^{(I)}| |y^{(J)}| } \diff x\Bigg)^i \diff y.
	\end{align}
	If $I=\emptyset$, first using \eqref{eq:logint}, then splitting the exponential
	into the product of two exponentials with the power halved, using $a^i e^{-a}
	\le i!$ for $a\geq0$, and finally using Jensen's inequality and referring to \eqref{eq:3.1} yield that there exists a constant $C \in (0,\infty)$ such that
	\begin{align*}
	&s\int_{\XX}  \left(s \int_\XX
	\mathds{1}_{x \prec y} \|y\|^\alpha \big| \log (s|y|)\big|^\delta  \big| \log (s|x|)\big|^{\delta'} (s|x|)^\tau e^{-s |y|} \diff x\right)^i \diff y \\
	&\le  C s\int_{\XX} \|y\|^{i \alpha} (s|y|)^{i(1+\tau)} e^{-i s |y|}  \left(\big| \log (s|y|)\big|^{\delta} + \log (s|y|)\big|^{\delta + \lceil\delta'\rceil + d-1}\right)^i \diff y
	=\mathcal{O}(\log^{d-2} s).
	\end{align*}
	Similarly, when $J=\emptyset$, then Lemma~\ref{lem:intbd'} yields
	\begin{displaymath}
	s\int_{\XX}  \left(s \int_\XX
	\mathds{1}_{x \succ y} \|x\|^\alpha \big| \log (s|x|)\big|^{\delta+ \delta'} (s|x|)^\tau e^{-s |x|} \diff x\right)^i \diff y 
	=\mathcal{O}(\log^{d-2} s).
	\end{displaymath}
	Next, assume that $I$ is nonempty and of cardinality $\ell$, with
	$1 \le \ell \le d-1$.
	Using that $\|x \vee y\|^\alpha \le 2^\alpha (\|x^{(I)}\|^\alpha + \|y^{(J)}\|^\alpha)$ along with Jensen's inequality,
	\begin{align}\label{eq:spl}
	& 2^{-i \alpha - i+1} s\int_{\XX}  \Bigg(s \int_\XX \mathds{1}_{x^{(I)} \succ y^{(I)}, x^{(J)} \prec y^{(J)}} \|x \vee y\|^\alpha
	 \big| \log (s|x^{(I)}| |y^{(J)}|)\big|^\delta \nonumber\\
	 & \qquad \qquad \qquad \qquad\qquad \qquad \qquad \qquad\qquad \qquad\times \big| \log (s|x|)\big|^{\delta'} (s|x|)^\tau e^{-s |x^{(I)}|\,|y^{(J)}|} \diff x\Bigg)^i \diff y \nonumber\\
	&\le \Bigg[ s\int_{\XX}  \Bigg(s \int_\XX \mathds{1}_{x^{(I)} \succ y^{(I)}, x^{(J)} \prec y^{(J)}} \|x^{(I)}\|^\alpha
	 \big| \log (s|x^{(I)}| |y^{(J)}|)\big|^\delta \big| \log (s|x|)\big|^{\delta'} (s|x|)^\tau e^{-s |x^{(I)}|\,|y^{(J)}|} \diff x\Bigg)^i \diff y\nonumber\\
	&+ s\int_{\XX} \Bigg(s \int_\XX \mathds{1}_{x^{(I)} \succ y^{(I)}, x^{(J)} \prec y^{(J)}} \|y^{(J)}\|^\alpha
	 \big| \log (s|x^{(I)}| |y^{(J)}|)\big|^\delta \big| \log (s|x|)\big|^{\delta'} (s|x|)^\tau e^{-s |x^{(I)}|\,|y^{(J)}|} \diff x\Bigg)^i \diff y\Bigg].
	\end{align}
	Using \eqref{eq:logint} for the $d-\ell$ dimensional unit cube for the inequality, we have for some $C,C' \in (0,\infty)$ that
	\begin{align*}
	&s \int_\XX \mathds{1}_{x^{(I)} \succ y^{(I)}, x^{(J)} \prec y^{(J)}} \|x^{(I)}\|^\alpha
	\big| \log (s|x^{(I)}| |y^{(J)}|)\big|^\delta \big| \log (s|x|)\big|^{\delta'}(s|x|)^\tau e^{-s |x^{(I)}|\,|y^{(J)}|}  \diff x\\
	&= \int_{[0,1]^\ell}  \mathds{1}_{x^{(I)} \succ y^{(I)}}  \|x^{(I)}\|^\alpha \big| \log (s|x^{(I)}| |y^{(J)}|)\big|^\delta  e^{-s |x^{(I)}|\,|y^{(J)}|} \\
	& \qquad \qquad \qquad \qquad \qquad \times \left(s\int_{[0,1]^{d-\ell}} \mathds{1}_{x^{(J)} \prec y^{(J)}}  \big| \log (s|x^{(I)}|\,|x^{(J)}|)\big|^{\delta'} (s|x^{(I)}|\,|x^{(J)}|)^\tau \diff x^{(J)}\right) \diff x^{(I)}\\
	&\le \int_{[0,1]^\ell}  \mathds{1}_{x^{(I)} \succ y^{(I)}}  \|x^{(I)}\|^\alpha \big| \log (s|x^{(I)}| |y^{(J)}|)\big|^\delta e^{-s |x^{(I)}|\,|y^{(J)}|} \\
	& \qquad \qquad \qquad \qquad \qquad \times \left(C s|y^{(J)}|\left[1+\big|\log s|x^{(I)}||y^{(J)}|\big|^{d-\ell-1+\lceil\delta'\rceil}\right]\right) (s|x^{(I)}||y^{(J)}|)^\tau \diff x^{(I)}\\
	&= C s|y^{(J)}| \int_{[0,1]^\ell}  \mathds{1}_{x^{(I)} \succ y^{(I)}}  \|x^{(I)}\|^\alpha e^{-s |x^{(I)}|\,|y^{(J)}|}\\
	& \qquad \qquad \qquad\qquad \qquad \times\left[ \big| \log (s|x^{(I)}| |y^{(J)}|)\big|^\delta + \big|\log s|x^{(I)}||y^{(J)}|\big|^{d-\ell-1 +\delta+\lceil\delta'\rceil} \right]  (s|x^{(I)}||y^{(J)}|)^\tau \diff x^{(I)}\\
	&\le C' \frac{\|y^{(I)}\|^{\alpha'}}{(s|y|)^{\alpha'}} e^{-s|y|/2}\Big[1+\big|\log(s|y|)\big|^{(\ell-2)^+ + \mathds{1}_{l=1}}\Big],
	\end{align*}
	for some $\alpha' \in (0,\alpha]$ such that $i\alpha'<1$, where the last step is due to Corollary~\ref{cor:c-bound}. Hence, plugging this bound in, followed by Jensen's inequality for the first step, we have that for some $C''\in (0,\infty)$,
	\begin{align*}
	& s\int_{\XX}  \left(s \int_\XX \mathds{1}_{x^{(I)} \succ y^{(I)}, x^{(J)} \prec y^{(J)}} \|x^{(I)}\|^\alpha
	 \big| \log (s|x^{(I)}| |y^{(J)}|)\big|^\delta \big| \log (s|x|)\big|^{\delta'} (s|x|)^\tau e^{-s |x^{(I)}|\,|y^{(J)}|} \diff x\right)^i \diff y\\
	&\le C'' s \int_{\XX}\frac{\|y^{(I)}\|^{i\alpha'}}{(s|y|)^{i\alpha'}} e^{-i
		s|y|/2}\Big[1+\big|\log(s|y|)\big|^{i((\ell-2)^+ + \mathds{1}_{l=1})}\Big]\diff y =\mathcal{O}(\log^{d-2} s),
	\end{align*}
	where we have used the trivial bound $\|y^{(I)}\| \le \|y\|$ and \eqref{eq:3.1} for the final step. A similar argument for the second summand on the right-hand side of \eqref{eq:spl} using Lemma~\ref{lemma:c-bound} and \eqref{eq:3.1} gives
	\begin{align*}
	& s\int_{\XX} \left(s \int_\XX \mathds{1}_{x^{(I)} \succ y^{(I)}, x^{(J)} \prec y^{(J)}} \|y^{(J)}\|^\alpha
	 \big| \log (s|x^{(I)}| |y^{(J)}|)\big|^\delta \big| \log (s|x|)\big|^{\delta'} (s|x|)^\tau e^{-s |x^{(I)}|\,|y^{(J)}|} \diff x\right)^i \diff y\\
	& \le  C s \int_{\XX}\|y^{(J)}\|^{i \alpha} e^{-i
		s|y|/2}\Big[1+\big|\log(s|y|)\big|^{i(\ell-1)}\Big]\diff y=\mathcal{O}(\log^{d-2} s).
	\end{align*}
	The bound in the first assertion now follows from \eqref{eq:maxsplit}.
	Finally, the second assertion of the lemma follows from the first one with $\delta'=0$ upon using Lemma~\ref{lemma:c-bound} and Jensen's inequality. 
\end{proof}

Now we are ready to derive the bound in
Theorem~\ref{thm:KolBd} for $H_s=\mathscr{L}_0^\alpha(\sP_{s})$. 
Recall the constants $\theta=p /(32+4 p)$
and $\lambda=p/(40+10p)$. For our example, it suffices to let
$p=1$. Nonetheless, the bounds in the following three lemmas are derived for any positive $\theta$
and $\lambda$.

\begin{lemma}
	\label{lem:intg1s}
	For $d \ge 2$, $s>1$, $\theta>0$, $\lambda>0$ and $f_{2\theta}$
	defined as in \eqref{eq:fa},
	\begin{align*}
	s \int_{\XX} f_{2\theta}(x) \diff x=\mathcal{O}(\log^{d-2}s).
	\end{align*}
\end{lemma}

\begin{proof} 
	We first bound the integral of $f_{2\theta}^{(1)}$ defined at \eqref{eq:fal}. Recall, $\widetilde{M}_{s}(x)=\max\{\|x\|^{2\alpha},\|x\|^{4\alpha}\}$. In the proof, we consider a generic exponent $t \in \{2\alpha, 4\alpha\}$ for the norm. Similarly, since $\tilde h_s = \max\{h_s^{2/(4+p/2)}, h_s^{4/(4+p/2)}\}$, we will consider a generic exponent $t' \in \{2/(4+p/2), 4/(4+p/2)\}$ for $h_s$. By Lemma~\ref{lem:intbd'}, we obtain
	\begin{equation}\label{eq:d}
	s\int_{\XX} s\int_{\XX} \|y\|^t e^{-2\theta r_s(y,x)}\diff y
	\diff x =  s \int_{\XX} s \int_\XX \mathds{1}_{y \succ x} \|y\|^t
	e^{- 2\theta s |y|}\diff y
	\diff x
	=\mathcal{O}(\log^{d-2} s). 
	\end{equation}
	Recall $g_{s}(y)=c_{\lambda,s}(y) \equiv c(y)$. Since $2\theta s|y|e^{- 2\theta s |y|} \le 1$, using Corollary~\ref{cor:c-bound} with $\alpha'\in (0,4+p/2]$ such that $1-t'\alpha' \ge 1-4 \alpha'/(4+p/2)>-1$,
	Lemma~\ref{lem:BM21} and Jensen's inequality for the second step, and \eqref{eq:3.1} for the final, there exists a constant $C \in (0,\infty)$ such that
	\begin{align}\label{eq:2.1}
	s &\int_{\XX} s \int_{\XX} h_s(y)^{t'} (1+g_{s}(y)^4) e^{- 2\theta r_s(y,x)}  \diff y \diff x
	= \frac{s}{\lambda^{t'}} \int_{\XX} s|y| \bar c (y)^{t'} (1+c(y)^4)
	e^{- 2\theta s |y|}  \diff y \nonumber \\
	&\le Cs \int_{\XX}
	\frac{\|y\|^{t'\alpha'}}{(s|y|)^{t'\alpha'-1}} e^{-t' \lambda s|y|/2}\Big[1+\big|\log(\lambda s|y|)\big|^{t'(d-2)}\Big] \left(1+e^{-4\lambda s|y|/2}\Big[1+\big|\log(\lambda s|y|)\big|^{4(d-1)}\Big]\right) \diff y \nonumber\\
	&=\mathcal{O}(\log^{d-2} s).
	\end{align}
	Combining \eqref{eq:d} and \eqref{eq:2.1}, we obtain
	\begin{displaymath}
	s \int_{\XX} f_{2\theta}^{(1)}(x) \diff x =\mathcal{O}(\log^{d-2} s).
	\end{displaymath}
	We move on to $f_{2\theta}^{(2)}$. Using again that $xe^{-x} \le 1$
	for $x \geq0$ and \eqref{eq:3.1}, we have
	\begin{align*}
	s\int_{\XX} &s\int_{\XX} \|y\|^t e^{-2\theta r_s(x,y)}\diff y
	\diff x \le  s^2 \int_{\XX} \|x\|^t \int_\XX
	\mathds{1}_{y \prec x} e^{- 2\theta s |x|}
	\diff y \diff x\\
	&= s \int_{\XX} s|x|  \|x\|^t e^{- 2\theta s |x|}\diff x
	\le  s\theta^{-1} \int_{\XX} \|x\|^t e^{- \theta s |x|}\diff x=\mathcal{O}(\log^{d-2} s). 
	\end{align*}
	Also, letting $\lambda'=\min\{\lambda,2\theta\}$ and noting that $c_{\lambda,s}(y)=c (y) \le c_{\lambda',s}(y)$,
	\begin{align*}
	s \int_{\XX} s \int_{\XX} h_s(y)^{t'} (1+g_{s}(y)^4) e^{- 2\theta r_s(x,y)} 
	\diff y \diff x
	&= \frac{s}{\lambda^{t'}} \int_{\XX} \bar c (y)^{t'} (1+c(y)^4) \left(s\int_\XX \mathds{1}_{x \succ y}
	e^{-2\theta s |x|} \diff x \right)\diff y\\
	&\le  \frac{s}{\lambda^t} \int_{\XX} \bar c (y)^{t'} (c_{\lambda',s}(y) + c_{\lambda',s}(y)^5)\diff y =\mathcal{O}(\log^{d-2} s),
	\end{align*}
	where the last step follows similarly as in \eqref{eq:2.1}. Thus,
	\begin{displaymath}
	s \int_{\XX} f_{2\theta}^{(2)}(x) \diff x=\mathcal{O}(\log^{d-2} s).
	\end{displaymath}
	It remains to  bound the integral of $f_{2\theta}^{(3)}$. Using Lemma~\ref{lem:BM21}, the inequality $(a+b)^\theta \le 2^\theta (a^\theta + b^\theta)$ for $a, b,\theta \ge 0$ and Lemma~\ref{lem:maxlog},
	\begin{multline*}
	s \int_{\XX} s \int_{\XX} \|y\|^t q_s(x,y)^{2\theta}
	\diff y \diff x = s \int_{\XX} s \int_{\XX} \|y\|^t c_{1,s}(x \vee y)^{2\theta} 
	\diff y \diff x\\
	\le C s^2 \int_{\XX^2} \|y\|^t e^{-\theta s |x \vee y|} \left(1+ |\log (s|x \vee y|)|^{2\theta(d-1)}\right)
	\diff (x,y)=\mathcal{O}(\log^{d-2} s).
	\end{multline*}
	Finally, again using Lemma~\ref{lem:BM21} for the inequality, a similar argument as above with $\alpha'\in (0,4+p/2]$ such that $t'\alpha' <1$ yields
	\begin{align}\label{eq:7}
	&s \int_{\XX} s \int_{\XX} h_s(y)^{t'} (1+g_s(y)^4) q_s(x,y)^{2\theta}
	\diff y \diff x\nonumber\\
	&= \frac{s}{\lambda^{t'}} \int_{\XX}  s \int_{\XX}  \bar c (y)^{t'}  (1+c(y)^4)
	c_{1,s}(x \vee y)^{ 2\theta} \diff y \diff x \nonumber\\
	&\le C s \int_{\XX}  s \int_{\XX} \frac{\|y\|^{t'\alpha'}}{(s|y|)^{t'\alpha'}} \left(1+|\log (\lambda s|y|)|^{d-1}\right)^{4+t'}
	e^{-\theta s |x \vee y|} \left(1+|\log (s|x \vee y|)|^{d-1}\right)^{ 2\theta} \diff y \diff x\nonumber\\
	&=\mathcal{O}(\log^{d-2} s),
	\end{align}
	where we have used Lemma~\ref{lem:maxlog} for the final step. Combining the above two bounds, we obtain 
	\begin{displaymath}
	s \int_{\XX} f_{2\theta}^{(3)}(x) \diff x=\mathcal{O}(\log^{d-2} s).
	\end{displaymath}
	Putting together the bounds for the integrals of $f_{2\theta}^{(i)}$ for $i=1,2,3$ concludes the proof.
\end{proof}

\begin{lemma}\label{lem:intg2s}
	For $d \ge 2$, $s>1$, $\theta>0$, $\lambda>0$ and $f_{\theta}$ defined as in \eqref{eq:fa},
	\begin{align*}
	s \int_{\XX} f_{\theta}(x)^2 \diff x =\mathcal{O}(\log^{d-2} s).
	\end{align*}
\end{lemma}

\begin{proof}
	As in Lemma~\ref{lem:intg1s}, we consider integrals of squares of
	$f_\theta^{(i)}$ for $i=1,2,3$ separately. Again, we take generic exponents $t \in \{2\alpha, 4\alpha\}$ and $t' \in \{2/(4+p/2), 4/(4+p/2)\}$ for the norm, and for $h_s$, respectively. By Lemma~\ref{lem:intbd'}, for any $t>0$,
	\begin{displaymath}
	s\int_{\XX} \left(s\int_{\XX} \|y\|^t e^{-\theta r_s(y,x)}\diff y\right)^2
	\diff x =\mathcal{O}(\log^{d-2} s). 
	\end{displaymath}
	Using Lemma~\ref{lem:BM21} and Corollary~\ref{cor:c-bound} with $\alpha'\in (0,4+p/2]$ such that $t'\alpha' <1$ in the penultimate step, Jensen's inequality followed by Lemma~\ref{lem:intbd'} yields that for any $t >0$ there exists $C \in (0,\infty)$ such that
	\begin{align*}
	&s \int_{\XX} \left(s \int_{\XX}h_s(y)^{t'} (1+g_s(y)^4)  e^{- \theta r_s(y,x)} 
	\diff y\right)^2 \diff x\\
	&= \frac{s}{\lambda^{2t'}} \int_{\XX} \left(s \int_{\XX} \mathds{1}_{y \succ x}\bar c (y)^{t'}  (1+c(y)^4) e^{-\theta s|y|} 
	\diff y\right)^2 \diff x\\
	&\le  C s \int_{\XX} \left(s \int_{\XX} \mathds{1}_{y \succ x}\|y\|^{t'\alpha'} (s|y|)^{-t'\alpha'}\left(1+ |\log (\lambda s|y|)|^{d-1}\right)^{4+t'} e^{- \theta s|y|} \diff y\right)^2 \diff x 
	\le \mathcal{O}(\log^{d-2} s).
	\end{align*}
	Combining using Jensen's inequality, we obtain
	\begin{displaymath}
	s \int_{\XX} f_{\theta}^{(1)}(x)^2 \diff x =\mathcal{O}(\log^{d-2} s).
	\end{displaymath}
	Next, we integrate the square of $f_{\theta}^{(3)}$. Using
	Lemmas~\ref{lem:BM21} and \ref{lem:maxlog}, there exists a constant $C \in (0,\infty)$ such that
	\begin{align*}
	&s \int_{\XX} \left(s \int_{\XX} \|y\|^t q_s(x,y)^{\theta}
	\diff y\right)^2 \diff x
	= s \int_{\XX} \left(s \int_{\XX}  \|y\|^t c_{1,s}(x\vee y)^{\theta}
	\diff y\right)^2 \diff x\nonumber\\
	&\le Cs \int_{\XX}\left(s \int_\XX \|y\|^t e^{-\theta s |x\vee y|/2} 
	\diff y \right)^2 \diff x 
	+ Cs \int_{\XX} \left(s \int_\XX \|y\|^t e^{-\theta s |x\vee y|/2} |\log (s|x \vee y|)|^{\theta(d-1)}
	\diff y \right)^2 \diff x\nonumber\\
	&=\mathcal{O}(\log^{d-2} s).
	\end{align*}
	Arguing as for \eqref{eq:7}, bounding $c(y)$ and $c_{1,s}(x \vee y)$ using Lemma~\ref{lem:BM21}, $\bar c (y)$ using Corollary~\ref{cor:c-bound}, the inequality $(a+b)^\theta \le 2^\theta (a^\theta + b^\theta)$ for $a, b, \theta \ge 0$, Lemma~\ref{lem:maxlog} along with Jensen's inequality yield
	\begin{align}\label{eq:add}
	&s \int_{\XX} \left(s \int_{\XX} h_s(y)^{t'} (1+g_s(y)^4) q_s(x,y)^{\theta}
	\diff y \right)^2 \diff x\nonumber\\
	&=  \frac{s}{\lambda^{2t'}} \int_{\XX} \left( s \int_{\XX} \bar c (y)^{t'}  (1+c(y)^4)
	c_{1,s}(x \vee y)^{\theta} \diff y \right)^2 \diff x=\mathcal{O}(\log^{d-2} s).
	\end{align}
This implies
		\begin{displaymath}
	s \int_{\XX} f_{\theta}^{(3)}(x)^2 \diff x =\mathcal{O}(\log^{d-2} s).
	\end{displaymath}
	Finally, for the integral of $(f_{\theta}^{(2)})^2$, using that $a^2 e^{-a} \le 2$ for
	$a\geq0$ and Corollary~\ref{cor:1}, we have 
	\begin{multline*}
	s\int_{\XX} \left(s\int_{\XX} \|y\|^t e^{-\theta r_s(x,y)}\diff y \right)^2
	\diff x
	= s \int_{\XX} \left(s \int_\XX \mathds{1}_{y \prec x} \|y\|^{t} e^{-\theta s |x|} 
	\diff y \right)^2 \diff x\\
	\le s/\theta^2 \int_{\XX} \|x\|^{2t} (\theta s|x|)^2\; e^{- 2\theta s |x|}\diff x
	\le 2 s/\theta^2 \int_{\XX} \|x\|^{2t} e^{-\theta s |x|}\diff x
	=\mathcal{O}(\log^{d-2} s). 
	\end{multline*}
	Using the
	Cauchy--Schwarz inequality,
	Lemma~\ref{lem:intbd} yields that for $t>0$,
	\begin{align*}
	&s\int_{\XX} \left(s\int_{\XX}h_s(y)^{t'} (1+g_s(y)^4) e^{-\theta r_s(x,y)}\diff y \right)^2
	\diff x\\
	&=\frac{s}{\lambda^{2t'}} \int_{\XX} \left(s \int_\XX \mathds{1}_{y \prec x} \bar c(y)^{t'} (1+c(y)^4) 
	e^{-\theta s |x|} \diff y\right)^2 \diff x\\
	&= \frac{s^2}{\lambda^{2t'}}  \int_{\XX^2} \bar c(y^{(1)})^{t'} (1+c(y^{(1)})^4)\, \bar c(y^{(2)})^{t'} (1+c(y^{(2)})^4) \, c_{2\theta,s}(y^{(1)}
	\vee y^{(2)}) 
	\diff (y^{(1)},y^{(2)})\\
	 &\le \frac{1}{\lambda^{2t'}} \left(s\int_\XX \bar c(y^{(1)})^{2t'} (1+c(y^{(1)})^4)^2
	\diff y^{(1)}\right)^{1/2} \\
	& \qquad \qquad \times \left(s \int_\XX \left( s\int_\XX \bar c(y^{(2)})^{t'} (1+c(y^{(2)})^4) c_{2\theta,s}(y^{(1)}
	\vee y^{(2)}) \diff y^{(2)}\right)^2 \diff y^{(1)}\right)^{1/2} =\mathcal{O}(\log^{d-2} s),
	\end{align*}
	where for the final step, the first factor is bounded similarly as in \eqref{eq:2.1} while for second, we argue as in \eqref{eq:add}. Thus,
	\begin{displaymath}
	s \int_{\XX} f_{\theta}^{(2)}(x)^2 \diff x =\mathcal{O}(\log^{d-2} s).
	\end{displaymath}
	Combining the bounds on the integrals of $f_{\theta}^{(i)}(x)^2$ for $i=1,2,3$ using Jensen's inequality, we obtain the desired result.	
\end{proof}

\begin{lemma}\label{lem:qgint} 
	For $d \ge 2$, $s>1$, $\theta>0$ and $\lambda>0$, let $G_s$ and $\kappa_s$ be as in
	\eqref{eq:g5} and \eqref{eq:p}, respectively. Then
	\begin{displaymath}
	s\int_\XX G_s(x)\big(\kappa_s(x)+g_{s}(x)\big)^{2\theta} \diff x
	=\mathcal{O}(\log^{d-2} s).
	\end{displaymath}
\end{lemma}

\begin{proof}
	First note that
	\begin{displaymath}
	\kappa_s(x)=\Prob{\xi_{s}(x, \sP_{s}+\delta_x)\neq 0}=e^{-s |x|}, \quad x \in \XX.
	\end{displaymath}
	Corollary~\ref{cor:1} and an argument as in \eqref{eq:2.1} yield that for any $t,t'>0$,
	\begin{align*}
	s \int_\XX
	\left[\|x\|^{t} + \bar c (x)^{t'} (1+c(x)^4)\right] e^{-2\theta s |x|} \diff x
	\le \mathcal{O}(\log^{d-2} s),
	\end{align*}
	which proves $s\int_\XX G_s(x)\kappa_s(x)^{2\theta} \diff x=\mathcal{O}(\log^{d-2} s)$.
	Repeating a similar argument, one also obtains
 	 $s\int_\XX G_s(x) g_{s}(x)^{2\theta} \diff x=\mathcal{O}(\log^{d-2} s)$. By an application of Jensen's inequality, we obtain the desired conclusion.
\end{proof}

\begin{proof}[Proof of Theorem~\ref{thm:Pareto}:] 
	By Theorem~\ref{thm:MV}(b), there exists $C_1 \in (0,\infty)$ such that $\Var(\mathscr{L}_0^{\alpha}) \ge C_1 \log^{d-2} s$ for all
	$s > 1$. An application of Theorem~\ref{thm:KolBd} for $H_s(\sP_s)=\mathscr{L}_0^\alpha(\sP_s)$ with
	Lemmas~\ref{lem:intg1s}, \ref{lem:intg2s} and \ref{lem:qgint} now
	yields the result.
\end{proof}

\subsection*{Acknowledgements} The author would like to thank Andrew Wade for some very helpful comments, and the two anonymous referees for their valuable suggestions that vastly improved the presentation and clarity of the manuscript.

\bibliography{MDST_RSA}
\bibliographystyle{alpha}

\end{document}